\documentclass[12pt]{article}

\usepackage[T1]{fontenc}
\usepackage[active]{srcltx}
\usepackage{lmodern}
\usepackage{amsmath,amstext,amssymb,graphicx,graphics,color}
\usepackage{mathtools,mathrsfs}
\usepackage{tikz}
\usetikzlibrary{decorations.markings,calc,fit,intersections,arrows,matrix,trees,positioning,shapes,tikzmark,backgrounds}
\usepackage{mleftright}
\usepackage{pgfplots}
\usepackage{theorem}
\usepackage{cite}               
\usepackage{hyperref}           
\usepackage{bbm}                
\usepackage{fancybox}           
\usepackage[section]{placeins}  
\usepackage[font=footnotesize]{caption}
\usepackage{subcaption}

\usepackage{arydshln} 
\tikzset{
    invisible/.style={opacity=0},
    visible on/.style={alt={#1{}{invisible}}},
    alt/.code args={<#1>#2#3}{%
      \alt<#1>{\pgfkeysalso{#2}}{\pgfkeysalso{#3}}%
  }
}
\theorembodyfont{\normalfont\slshape} 
\newtheorem{definition}{Definition}
\newtheorem{theorem}{Theorem}
\newtheorem{proposition}{Proposition}[section]
\newtheorem{corollary}[proposition]{Corollary}

\newtheorem{lemma}[proposition]{Lemma}
\theoremstyle{break} 

\newenvironment{proof}%
{{\par\noindent \bf Proof. \nobreak}}%
{\nobreak \removelastskip \nobreak \hfill $\Box$ \medbreak}
{{\par\noindent \bf Proof \nobreak}}%
{\nobreak \removelastskip \nobreak \hfill $\Box$ \medbreak}
{{\par\noindent \bf Proof lemma. \nobreak}}%
{\nobreak \removelastskip \nobreak \bf End proof lemma. \medbreak}
\newenvironment{remark}{\par \medskip \noindent {\bf Remark. }\nobreak}{\par \medskip}
\usepackage{fancyhdr}
 \fancypagestyle{plain}{
  \fancyhead{}
  
}

\def\paragraph#1{{\bf #1\ }}
\newcommand{\expo}{\mathrm{e}}

\newcommand{\Var}{\mathrm{Var}}

\newcommand{\dd}{\mathrm{d}}

\newcommand{\HH}{\mathrm{H}}

\newcommand{\overbar}[1]{\mkern 1.5mu\overline{\mkern-1.5mu#1\mkern-1.5mu}\mkern 1.5mu}

\usepackage[utf8]{inputenc}
\usepackage{unicode_character}

\title{Sticky dispersion on the complete graph: a kinetic approach}

\author{Fei Cao \footnotemark[1] \and Sebastien Motsch \footnotemark[2]}

\begin{document}
\maketitle

\footnotetext[1]{University of Massachusetts Amherst - Department of Mathematics and Statistics, Amherst, MA 01003, USA}
\footnotetext[2]{Arizona State University - School of Mathematical and Statistical Sciences, 900 S Palm Walk, Tempe, AZ 85287, USA}

\tableofcontents

\begin{abstract}
We study a variant of the dispersion process on the complete graph introduced in the recent work \cite{de_dispersion_2023} under the mean-field framework. We adopt a kinetic perspective (as opposed to the probabilistic approach taken in \cite{de_dispersion_2023} and many other related works) thanks to the reinterpretation of the model in terms of a novel econophysics model. Various analytical and quantitative results regarding the large time behaviour of the mean-field dynamics are obtained and supporting numerical illustrations are provided.
\end{abstract}


\section{Introduction}\label{sec:sec1}
\setcounter{equation}{0}


We study a modified version of the so-called dispersion process on the complete graph with $N$ vertices (denoted by $K_N$) introduced by Cooper, McDowell, Radzik, Rivera and Shiraga \cite{cooper_dispersion_2018} and investigated in many subsequent works \cite{de_dispersion_2023,frieze_note_2018,shang_longest_2020}. The set-up of the original dispersion process proposed in \cite{cooper_dispersion_2018,de_dispersion_2023} is as follows: Initially, $M \in \mathbb{N}_+$ indistinguishable particles are placed on a vertex of $K_N$. At the beginning of each time step, for every vertex occupied by at least two particles, each of these particles moves independently to another vertex on $K_N$ chosen uniformly at random. It is worth noting that the aforementioned process ends at the first time when each vertex host at most one particle.

In this manuscript, we plan to investigate a continuous-time variant of the aforementioned model which also embraces the effect of ``stickiness''. We first introduce some notations which will be used throughout this paper. We label the vertex set of $K_N$ from $1$ to $N$ and denote by $X_i(t)$ the number of particles inhabited by vertex $i$ at time $t \in \mathbb{R}_+$, and set \[X(t) = \left(X_1(t),\ldots,X_N(t)\right) \] to be the state vector of the dynamics. Our sticky version of the dispersion process is governed by the following dynamics: at random time (generalized by exponential law), each non-empty vertex $i$ expels a particle at the rate $R_i(t) = X_i(t)-1$ to another uniformly chosen vertex $j$. Formally, the state space is
\begin{equation}\label{eq:Omega}
\Omega = \{X \in \mathbb{N}^N \mid X_1 + \cdots + X_N = M\}.
\end{equation}

Before we dive into the mathematical analysis of the model defined above, we illustrate the intimate connection between our model to existing models in the literature through a sequence of remarks.

\begin{remark}
Notice that the original model proposed in \cite{cooper_dispersion_2018,de_dispersion_2023} can be recovered if we replace the dynamical description of our model by the following rule: at random time (generalized by exponential law), each non-empty vertex $i$ which is inhabited by at least two particles expels a particle at the rate $\tilde{R}_i(t) = X_i(t)$ to another uniformly chosen vertex $j$. Using the terminology introduced in \cite{de_dispersion_2023}, a particle is called \emph{happy} if it does not share the same vertex with other particles, otherwise it is \emph{unhappy}. Therefore, in the classical dispersion process on the complete graph $K_N$, happy particles have no motivation to make a move and each vertex which is inhabited by at least two particles expels a particle at a rate proportional to the number of particles located at that vertex. It is worth observing that happy particles also do not move in our sticky version of the classical dispersion process, and the term ``sticky'' comes from the following interpretation of the model: if we identify particles as (movable) agents/tenants and vertices as real estates/houses, and assume that the first agent (say, agent $i$ for instance) visiting an empty house/site $j$ occupies the site $j$ permanently and hence becomes a landlord, meaning that agent $i$ (``owner'' of the site $j$) will never intend to jump out of site $j$ (which he/she views as his/her own property and can only serve as a temporary accommodation for other agents), then such assumption is captured by the rate $R_j(t) = X_j(t)-1$ given in the definition of our sticky dispersion process. We illustrate our model via figure \ref{fig:sticky_dispersion} below.
\begin{figure}[!htb]
\centering
\includegraphics[width=.67\textwidth]{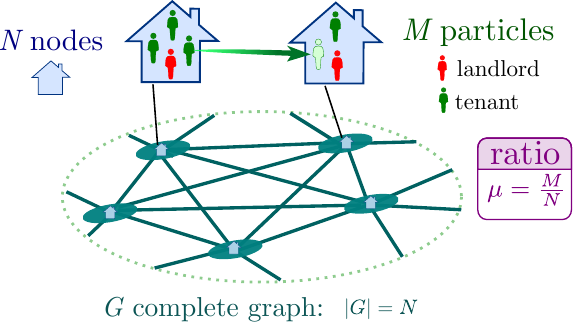}
\caption{Illustration of the sticky dispersion process on a complete graph with $N$ nodes (represented as houses) and $M$ particles (represented as tenants and landlords). The stickiness originates from the fact that landlord(s) will stay in their own houses forever and only tenants will jump across different houses.}
\label{fig:sticky_dispersion}
\end{figure}
\end{remark}

\begin{remark}
We can also reinterpret the sticky (and also the classical) dispersion process using terminologies from econophysics (which is a sub-branch of statistical physics that apply concepts and techniques of traditional physics to economics and finance \cite{dragulescu_statistical_2000,chakraborti_econophysics_2011,during_kinetic_2008,kutner_econophysics_2019,pereira_econophysics_2017,savoiu_econophysics_2013}). If we identity particles as dollars, vertices as agents, and $X_i(t)$ as the amount of dollars agent $i$ has at time $t$, then the aforementioned sticky dispersion process can be viewed as a simple dollar exchange mechanism in a closed economical system: at random time (generated by an exponential law), an agent $i$ who has at least two dollars in his/her pocket (i.e., $X_i \geq 2$) is picked at a rate proportional to $X_i - 1$, then he or she will give $1$ dollar to another agent $j$ picked uniformly at random. It is clearly from the set-up that we have
\begin{equation}\label{eq:preserved_sum}
X_1(t)+ \cdots + X_N(t) = N\mu = M \qquad \text{for all } t\geq 0
\end{equation}
as the economical system is closed, where $\mu = M \slash N$ represents the average amount of dollars per agent in this context. Mathematically, the update rule of this multi-agent system can be represented by
\begin{equation}\label{sticky_dispersion}
\textbf{sticky dispersion:} \qquad (X_i,X_j)~ \begin{tikzpicture} \draw [->,decorate,decoration={snake,amplitude=.4mm,segment length=2mm,post length=1mm}]
(0,0) -- (.6,0); \node[above,red] at (0.3,0) {\small{$X_i-1$}};\end{tikzpicture}~  (X_i-1,X_j+1) \quad (\text{if } X_i > 1).
\end{equation}
\end{remark}

We have noticed that earlier work on the classical dispersion process on complete graphs \cite{de_dispersion_2023} focuses only on the asymptotic region where the total number of particles $M$ scales no faster than $N$ (i.e., $\lim_{N \to \infty} = M \slash N \leq 1$), so that the process ends almost surely (which occurs when no particle is sharing the same vertex with other particles) so that every particle becomes happy at the (random) time $T_{K_N,M}$ termed as the \emph{dispersion time}. The main quantity under investigation in \cite{cooper_dispersion_2018,de_dispersion_2023,frieze_note_2018,shang_longest_2020} is the aforementioned dispersion time by virtue of advanced probabilistic tools. By resorting to a kinetic/mean-field approach, we intend to treat the case where $\mu = M \slash N \in (0,\infty)$ remains a positive constant of order $1$ and we also allow for general initial distributions of particles (beyond the very popular choice of putting all particles on a single vertex).

From now on, we will adopt the econophysics terminologies as mentioned in the previous remark (unless otherwise stated) and therefore identifying particles as dollars, $N$ vertices as $N$ agents, $X_i(t)$ as the amount of dollars agent $i$ has at time $t$, and distribution of particles over vertex set of $K_N$ as distribution of money among $N$ agents. We emphasize here that at least one advantage of adopting the econophysics interpretation of the model will be mentioned at the end of section \ref{subsec:sec3.1}. Our goal in this paper lies in the derivation of the limiting distribution of dollars among the agents as the total number of agents $N$ and time $t$ go to infinity, while keeping $\mu = M \slash N$ to be a fixed positive constant (which represents the average amount of dollars per agent in the system). To foresee the behavior of the dynamics under the large time and the large population limits, we perform agent-based simulations using $N = 5,000$ over $1,000,000$ steps using two different values for $\mu$ (see figure \ref{fig:agent_based_simulation}).

\begin{figure}[!htb]
  \centering
  \includegraphics[width=.97\textwidth]{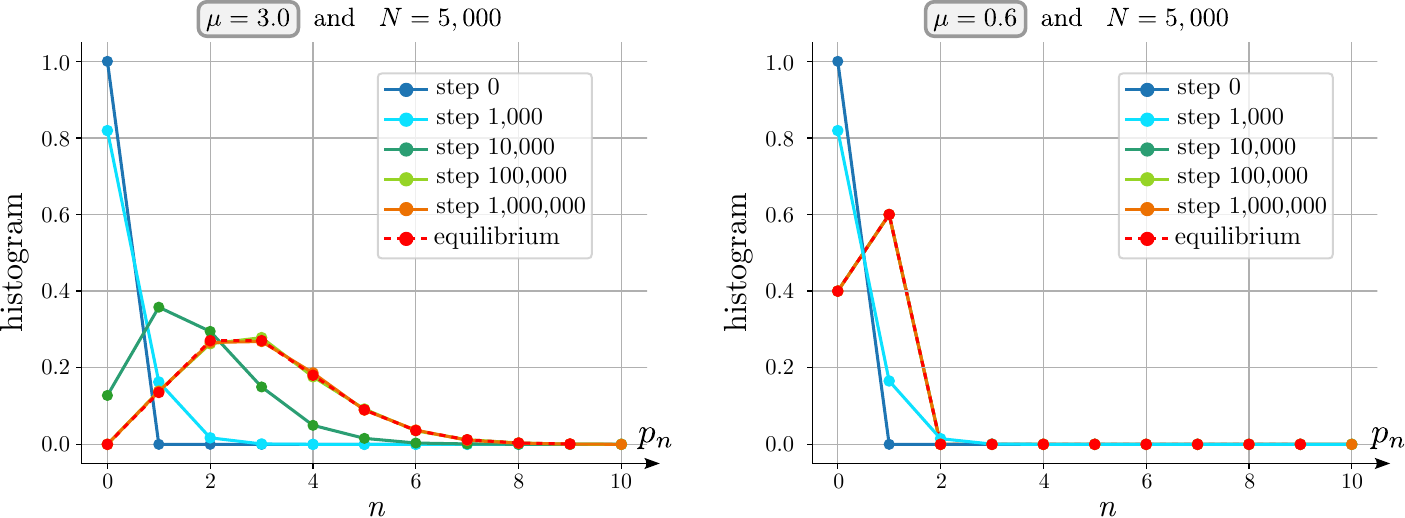}
  \caption{Distribution of money for the sticky dispersion model with $N = 5,000$ agents after $1,000,000$ steps, using two different values of $\mu$. Initially, we put all the money into the packet of a single agent (equivalently, all agents occupy the same house initially). {\bf Left:} For $\mu = 3$, the distribution approaches to a modified Poisson distribution which will be predicted by \eqref{eq:modified_Poisson} below. {\bf Right:} For $\mu = 0.6$, the distribution coincides with a Bernoulli distribution with mean $0.6$.}
  \label{fig:agent_based_simulation}
\end{figure}

We observe that the distribution of money converges to a Bernoulli distribution with mean $0.6$ for $\mu = 0.6$ (figure \ref{fig:agent_based_simulation}-right) and that it converges to a modified Poisson distribution predicated by \eqref{eq:modified_Poisson} for $\mu = 3$ (figure \ref{fig:agent_based_simulation}-left).

\subsection{Overview of main results}\label{subsec:sec1.1}

The model we have introduced is amenable to mean-field type analysis under the large population limit $N \to \infty$ which is detailed in section \ref{sec:sec2}, where the concept of \emph{propagation of chaos} \cite{sznitman_topics_1991} plays a crucial role. Bearing in mind the aim to obtain a simplified (actually fully deterministic) dynamics when we send $N \to \infty$, we consider the probability distribution function of money:
\begin{equation}
  \label{eq:p}
  {\bf p}(t)=\left(p_0(t),p_1(t),\ldots,p_n(t),\ldots\right)
\end{equation}
with $p_n(t)= \{``\text{probability that a typical agent has } n \text{ dollars at time}~ t "\}$. We will demonstrate in section \ref{sec:sec2} that evolution of ${\bf p}(t)$ is governed by the following deterministic system of nonlinear ordinary differential equations:
\begin{equation}
  \label{eq:law_limit}
  \frac{\dd}{\dd t} {\bf p}(t) = Q[{\bf p}(t)]
\end{equation}
with
\begin{equation}
  \label{eq:Q}
  Q[{\bf p}]_n = \left\{
    \begin{array}{ll}
      -\left(\sum_{k\geq 2} (k-1)\,p_k\right)\,p_0 & \quad n=0, \\
      n\,p_{n+1}+\left(\sum_{k\geq 2} (k-1)\,p_k\right)\,p_{n-1}- \left(n-1+\sum_{k\geq 2} (k-1)\,p_k\right)\,p_n, & \quad n \geq 1.
    \end{array}
  \right.
\end{equation}
The rigorous justification of this transition from the stochastic interacting agents systems \eqref{sticky_dispersion} into the corresponding mean-field ODE system \eqref{eq:law_limit}-\eqref{eq:Q} requires the proof of the \emph{propagation of chaos} property \cite{sznitman_topics_1991}, which is beyond the scope of the present manuscript and deserves a separate treatment on its own. On the other hand, propagation of chaos property has been proved for other econophysics models, see for instance \cite{cao_derivation_2021,cao_entropy_2021,cao_explicit_2021,cao_interacting_2022,cortez_uniform_2022,cortez_quantitative_2016} and we also refer interested readers to \cite{boghosian_h_2015,cao_biased_2023,cao_binomial_2022,cao_uncovering_2022,chakraborti_statistical_2000,chatterjee_pareto_2004,heinsalu_kinetic_2014,lanchier_rigorous_2017,lanchier_rigorous_2019,matthes_steady_2008} for many other interesting models in econophysics literature that we omit to describe in details.

Once the mean-field system of ODEs \eqref{eq:law_limit}-\eqref{eq:Q} associated to the multi-agent system has been identified, one natural follow-up step is to investigate the long time behaviour of the infinite dimensional ODE system \eqref{eq:law_limit}-\eqref{eq:Q} with the hope of showing convergence of its solution towards an equilibrium distribution, and we take on this task in section \ref{sec:sec3}.  Surprisingly, as will be shown in section \ref{sec:sec3}, the large time asymptotic of the solution to \eqref{eq:law_limit}-\eqref{eq:Q} depends on the value of the parameter $\mu \in (0,\infty)$ which represents the average amount of dollars each individual has initially. We prove in section \ref{subsec:sec3.1} that solutions of \eqref{eq:law_limit}-\eqref{eq:Q} converges to the following Bernoulli distribution ${\bf p}^* = \left(p^*_0,p^*_1,\ldots,p^*_n,\ldots\right)$
\begin{equation}\label{eq:Bernoulli}
p^*_0 = 1-\mu, \quad p^*_1 = \mu, \quad p^*_n = 0 \quad \text{for } n\geq 2
\end{equation}
when $\mu \in (0,1]$. Note in particular that the two-point Bernoulli distribution \eqref{eq:Bernoulli} boils down to the Dirac delta distribution $\delta_1$ centered at $1$, defined via
\begin{equation}\label{eq:Dirac_delta}
\delta_1 = (0,1,0,\ldots,0,\ldots),
\end{equation}
when $\mu = 1$. We demonstrate in section \ref{subsec:sec3.2} the convergence of solutions of \eqref{eq:law_limit}-\eqref{eq:Q} to the following modified Poisson distribution $\overbar{{\bf p}} = \left(\overbar{p}_0,\overbar{p}_1,\ldots,\overbar{p}_n,\ldots\right)$
\begin{equation}\label{eq:modified_Poisson}
\overbar{p}_0 = 0, \quad \overbar{p}_n = \frac{(\mu-1)^{n-1}}{(n-1)!}\,\expo^{-(\mu-1)} \quad \text{for } n\geq 1
\end{equation}
when $\mu > 1$. We remark here that the mathematical analysis of the large time behavior of the system \eqref{eq:law_limit}-\eqref{eq:Q} is tricker when $\mu > 1$ due to the lack of an appropriate Lyapunov functional (which decays monotonically as time progresses) associated to the system, although our analysis is significantly simplified when the initial condition ${\bf p}(0)$ is chosen such that $p_0(0) = 0$.

\section{Asymptotic nonlinear ODE system (as $N \to \infty$)}\label{sec:sec2}
\setcounter{equation}{0}

For the sake of completeness, we devote this section to a formal derivation of the mean-field system of nonlinear ODEs \eqref{eq:law_limit}-\eqref{eq:Q}, although the following heuristic argument is well-known to experts in mean-field type theory. Notice that the stochastic process $\{(X_1(t),\ldots,X_N(t)\}$ is a pure jump process with jumps of the form
\begin{equation*}
(X_1,\ldots,X_i,\ldots,X_j,\ldots,X_N) \to (X_1,\ldots,X_i\!-\!1,\ldots,X_j\!+\!1,\ldots,X_N)
\end{equation*}
with transition rate $\frac{X_i-1}{N}\,\mathbbm{1}_{[2,\infty)}\left(X_i\right)$. Therefore, taking $\phi(x_1,\ldots,x_N)$ to be a test function, the process satisfies the Kolmogorov backward equation
\begin{eqnarray}
&\frac{\dd}{\dd t} \mathbb E[\phi(X_1,\ldots,X_N)] = \frac{1}{N} \sum_{i,j=1}^N \mathbb E[\mathbbm{1}_{[2,\infty)}\left(X_i\right)(X_i\!-\!1)\phi(X_1\ldots,X_i\!-\!1,\ldots,X_j\!+\!1,\ldots,X_N)] \nonumber \\
&~-~ \mathbb E[\mathbbm{1}_{[2,\infty)}\left(X_i\right)(X_i\!-\!1)\phi(X_1,\ldots,X_N)]. \nonumber
\end{eqnarray}
Denoting by $\rho(x_1,\ldots,x_N)$ the probability distribution function associated with the process $\{(X_1,\ldots,X_N)\}$, and denote $\rho_1(x_1)$ and $\rho_2(x_1,x_2)$ as the first and the second marginal distributions of $\rho(x_1,\ldots,x_N)$, respectively. We deduce that $\rho$ satisfies the following Kolmogorov forward equation:
\begin{equation*}
\partial_t \rho(\mathbf{x}) = \frac{1}{N} \sum_{i,j=1}^N \mathbbm{1}_{[1,\infty)}\left(x_j\right)\,x_i\,\rho(\mathbf{x_{ij}}) - \mathbbm{1}_{[1,\infty)}\left(x_i\right)\,(x_i\!-\!1)\, \rho(\mathbf{x}),
\end{equation*}
with the notations $\mathbf{x}:=(x_1,\ldots,x_N)$ and $\mathbf{x_{ij}}:=(x_1,\ldots,x_i\!+\!1,\ldots,x_j\!-\!1,\ldots,x_N)$.

Now if we turn to the first marginal by supposing that the test function $\phi$ depends only on $x_1$, then
\begin{align*}
\frac{\dd}{\dd t}\, \mathbb E[\phi(X_1)] &= \mathbb E\Big[\left(\mathbbm{1}_{[2,\infty)}\left(X_1\right)(X_1\!-\!1)\phi(X_1\!-\!1)
                                      -\mathbbm{1}_{[2,\infty)}\left(X_1\right)(X_1\!-\!1)\phi(X_1) \right) \\
                                    &~~+ \frac{1}{N}\,\sum_{j=2}^N\,\left(\mathbbm{1}_{[2,\infty)}\left(X_j\right)(X_j\!-\!1)\phi(X_1\!+\!1) - \mathbbm{1}_{[2,\infty)}\left(X_j\right)(X_j\!-\!1)\phi(X_1)\right)\Big] \\
                                    &= \mathbb E\Big[\mathbbm{1}_{[2,\infty)}\left(X_1\right)(X_1\!-\!1)\phi(X_1\!-\!1) - \mathbbm{1}_{[2,\infty)}\left(X_1\right)(X_1\!-\!1)\phi(X_1) \\
                                    &~~+ \mathbbm{1}_{[2,\infty)}\left(X_2\right)(X_2\!-\!1)\phi(X_1\!+\!1) - \mathbbm{1}_{[2,\infty)}\left(X_2\right)(X_2\!-\!1)\phi(X_1)\Big],
\end{align*}
where the last identity follows since $\{X_i\}^N_{i=1}$ is interchangeable. The associated Kolmogorov forward equation now reads as
\begin{equation}\label{eq:KFE}
\begin{aligned}
  \partial_t \rho_1(x_1) &= x_1\,\rho_1(x_1\!+\!1)-\mathbbm{1}_{[1,\infty)}(x_1)(x_1\!-\!1)\rho_1(x_1) \\
  &~~ +\sum_{x_2 \geq 0} \mathbbm{1}_{[1,\infty)}(x_2)(x_2\!-\!1)\rho_2(x_1\!-\!1,x_2)\mathbbm{1}_{[1,\infty)}(x_1) - \mathbbm{1}_{[1,\infty)}(x_2)(x_2\!-\!1)\rho_2(x_1,x_2).
\end{aligned}
\end{equation}
In order to obtain a closed equation from \eqref{eq:KFE}, we assume the so-called \emph{propagation of chaos} property \cite{sznitman_topics_1991}, which implies that
\begin{equation}\label{eq:poc}
\rho_2(x_1,x_2)= \rho_1(x_1)\,\rho_1(x_2).
\end{equation}
Heuristically speaking, the propagation of chaos property \eqref{eq:poc} in the context of our model states that any pair of agents (say agent $i$ and agent $j$) picked to engage in the binary trading mechanism \eqref{sticky_dispersion} are statistically independent as $N \to \infty$. Equation \eqref{eq:poc} then allows us to deduce that
\begin{equation}\label{eq:sec2_eq1}
  \partial_t \rho_1(x_1) = x_1\,\rho_1(x_1\!+\!1)-\mathbbm{1}_{[1,\infty)}(x_1)(x_1\!-\!1)\rho_1(x_1) + \nu\,\mathbbm{1}_{[1,\infty)}(x_1)\,\rho_1(x_1\!-\!1) - \nu\,\rho_1(x_1),
\end{equation}
in which
\begin{equation}\label{eq:nu}
\nu = \sum_{x_2 \geq 1} (x_2-1)\,\rho_1(x_2).
\end{equation}
Since $x_1 \in \mathbb N$, it is readily seen that the system \eqref{eq:sec2_eq1} coincides with the mean-field ODE systems \eqref{eq:law_limit}-\eqref{eq:Q} introduced in section \ref{sec:sec1}.

\section{Large time behaviour (as $t \to \infty$)}\label{sec:sec3}
\setcounter{equation}{0}

After we achieved the transition from the interacting agents system \eqref{sticky_dispersion} to the deterministic nonlinear ODE system \eqref{eq:law_limit}-\eqref{eq:Q}, our main goal in this section is to show convergence of solution of \eqref{eq:law_limit}-\eqref{eq:Q} to its (unique) equilibrium solution. As we have indicated in the introduction, the large time behavior of solutions to \eqref{eq:law_limit}-\eqref{eq:Q} depends heavily on the range to which the parameter $\mu$ belongs. Before we dive into the detailed analysis of the system of nonlinear ODEs, we first establish some preliminary observations regarding solutions of \eqref{eq:law_limit}-\eqref{eq:Q}.
\begin{lemma}\label{lem:1}
If ${\bf p}(t)$ is a solution to the system \eqref{eq:law_limit}-\eqref{eq:Q}, then
\begin{equation}\label{eq:conservation_mass_mean_value}
\sum_{n=0}^\infty Q[{\bf p}]_n =0 \quad \text{and} \quad \sum_{n=0}^\infty n\,Q[{\bf p}]_n =0.
\end{equation}
In particular, the total probability mass and the average amount of dollars per agent are conserved.
\end{lemma}
The proof of Lemma \ref{lem:1} is based on straightforward computations and will be skipped. Thanks to these conservation relations, the solution ${\bf p}(t)$ lives in the space of probability distributions on $\mathbb N$ with the prescribed mean value $\mu$, defined by
\begin{equation}\label{eq:S_mu}
\mathcal{S}_\mu \coloneqq \{{\bf p} \mid \sum_{n=0}^\infty p_n =1,~p_n \geq 0,~\sum_{n=0}^\infty n\,p_n =\mu\}.
\end{equation}
More importantly, the system \eqref{eq:law_limit}-\eqref{eq:Q} will be equivalent to the following system of nonlinear ODEs:
\begin{equation}
  \label{eq:law_limit_repeat}
  \frac{\dd}{\dd t} {\bf p}(t) = Q[{\bf p}(t)]
\end{equation}
in which
\begin{equation}
  \label{eq:Q_rewrite}
  Q[{\bf p}]_n = \left\{
    \begin{array}{ll}
      -\left(\mu-1+p_0\right)\,p_0 & \quad n=0, \\
      n\,p_{n+1}+\left(\mu-1+p_0\right)\,p_{n-1}- (n-1)\,p_n - \left(\mu-1+p_0\right)\,p_n, & \quad n \geq 1.
    \end{array}
  \right.
\end{equation}

\begin{remark}
As is often the case, the Fokker-Planck equation \eqref{eq:law_limit_repeat}-\eqref{eq:Q_rewrite} admits a heuristic interpretation as a jump process with loss and gain, and we illustrate this point of view via figure \ref{fig:illustration_rates} below.
\begin{figure}[!htb]
\centering
\includegraphics[scale=1.0]{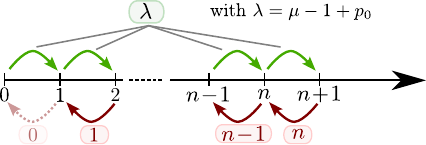}
\caption{Schematic illustration of the Fokker-Planck type system of nonlinear ODEs \eqref{eq:law_limit_repeat}-\eqref{eq:Q_rewrite}.}
\label{fig:illustration_rates}
\end{figure}
\end{remark}

Next, we identify the unique equilibrium solution associated with the system \eqref{eq:law_limit_repeat}-\eqref{eq:Q_rewrite}.

\begin{proposition}\label{prop:equilibrium}
The unique equilibrium solution of \eqref{eq:law_limit_repeat}-\eqref{eq:Q_rewrite} in the space $\mathcal{S}_\mu$, for $\mu \in (0,1]$, is given by
\begin{equation}\label{eq:Bernoulli_repeat}
p^*_0 = 1-\mu, \quad p^*_1 = \mu, \quad p^*_n = 0 \quad \text{for } n\geq 2.
\end{equation}
On the other hand, the unique equilibrium solution of \eqref{eq:law_limit_repeat}-\eqref{eq:Q_rewrite} in the space $\mathcal{S}_\mu$, when $\mu \in (1,\infty)$, is provided by
\begin{equation}\label{eq:modified_Poisson_repeat}
\overbar{p}_0 = 0, \quad \overbar{p}_n = \frac{(\mu-1)^{n-1}}{(n-1)!}\,\expo^{-(\mu-1)} \quad \text{for } n\geq 1.
\end{equation}
\end{proposition}

\begin{proof}
From the evolution equation defined by \eqref{eq:law_limit_repeat}-\eqref{eq:Q_rewrite}, it is straightforward to check that
\begin{equation}\label{eq:identity}
n\,p_{n+1} = \left(\sum_{k\geq 2} (k-1)\,p_k\right)\,p_n \quad \forall~n\geq 1,\quad \text{and} \quad \left(\sum_{k\geq 2} (k-1)\,p_k\right)\,p_0 = 0
\end{equation}
must hold at equilibrium. Thus for $\mu \leq 1$, we deduce that the unique equilibrium solution, denoted by ${\bf p}^*$, is
\begin{equation*}
p^*_0 = 1-\mu, \quad p^*_1 = \mu, \quad p^*_n = 0 \quad \text{for } n\geq 2.
\end{equation*}
On the other hand, for $\mu > 1$, we deduce from \eqref{eq:identity} that the unique equilibrium distribution, denoted by $\overbar{{\bf p}}$, is
\begin{equation*}
\overbar{p}_0 = 0, \quad \overbar{p}_n = \frac{(\mu-1)^{n-1}}{(n-1)!}\,\expo^{-(\mu-1)} \quad \text{for } n\geq 1.
\end{equation*}
We emphasize that both ${\bf p}^*$ and $\overbar{{\bf p}}$ belong to the space $\mathcal{S}_\mu$.
\end{proof}

\begin{remark}
It is worth noting that the equation governing the motion of $p_0(t)$ can be readily solved explicitly, leading us to
\begin{equation}\label{eq:p_0;case1}
p_0(t) = (1-\mu)\,\frac{p_0(0)\,\expo^{(1-\mu)\,t}}{p_0(0)\,\expo^{(1-\mu)\,t}+1-\mu-p_0(0)}
\end{equation}
for $\mu \neq 1$, and
\begin{equation}\label{eq:p_0;case2}
p_0(t) = \frac{1}{t+\frac{1}{p_0(0)}}
\end{equation}
when $\mu = 1$.
\end{remark}

As a warm-up before we dive into the analytical investigation of the nonlinear ODE system \eqref{eq:law_limit_repeat}-\eqref{eq:Q_rewrite}. We investigate numerically the convergence of ${\bf p}(t)$ to its equilibrium distribution. We use $\mu=3$ and $\mu = 0.6$ respectively. To discretize the model, we use $10,001$ components to describe the distribution ${\bf p}(t)$ (i.e., $(p_0(t),\ldots,p_{10000}(t))$). As initial condition, we use $p_{100}(0)= \frac{\mu}{100}$, $p_0(0) = 1-p_{100}(0)$ and $p_i(0)=0$ for $i \notin \{0,100\}$. The standard Runge-Kutta fourth-order scheme is used to discretize the ODE system \eqref{eq:law_limit_repeat}-\eqref{eq:Q_rewrite} with  the time step $\Delta t=0.001$.
We plot in figure \ref{fig:conv_to_equilibrium}-left and \ref{fig:conv_to_equilibrium}-right the evolution of the numerical solution ${\bf p}(t)$ at different times corresponding to $\mu=3$ and $\mu = 0.6$ respectively. It can be observed that convergence to equilibrium occur in both cases.

\begin{figure}[!htb]
  \centering
  \includegraphics[width=.97\textwidth]{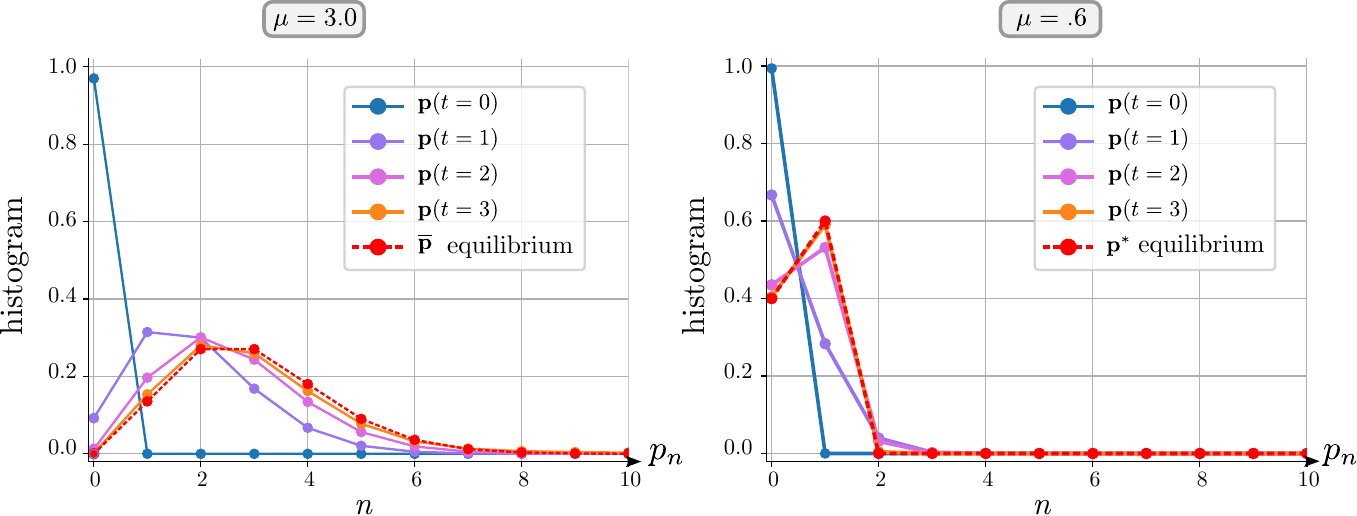}
  \caption{{\bf Left}: comparison between the numerical solution ${\bf p}(t)$ of the ODE system \eqref{eq:law_limit_repeat}-\eqref{eq:Q_rewrite} and the equilibrium $\overbar{{\bf p}}$ \eqref{eq:modified_Poisson} when $\mu = 3$. {\bf Right}: comparison between the numerical solution ${\bf p}(t)$ of the ODE system \eqref{eq:law_limit_repeat}-\eqref{eq:Q_rewrite} and the equilibrium ${\bf p}^*$ \eqref{eq:Bernoulli} when $\mu = 0.6$.}
  \label{fig:conv_to_equilibrium}
\end{figure}

\begin{remark}
The modified Poisson distribution $\overbar{{\bf p}}$ admits a simple interpretation in terms of random variables. Indeed, if $Y \sim \textrm{Possion}(\mu-1)$, we can easily deduce that $X \coloneqq 1 + Y$ is distributed according to the law $\overbar{{\bf p}}$.
\end{remark}

\subsection{Convergence to Bernoulli distribution for $\mu \leq 1$}\label{subsec:sec3.1}

To demonstrate the large time convergence of solutions of the system \eqref{eq:law_limit_repeat}-\eqref{eq:Q_rewrite} when $\mu \in (0,1]$, we rely on the construction of a suitable Lyapunov functional associated to the dynamics \eqref{eq:law_limit_repeat}-\eqref{eq:Q_rewrite}. We now recall the definition of Gini index.

\begin{definition}[\textbf{Gini index}]
For a given probability distribution function ${\bf p} \in \mathcal{P}(\mathbb N)$ with mean $\mu \in \mathbb{R}_+$, the Gini index of the distribution ${\bf p}$ is defined via
\begin{equation}\label{def1:Gini}
G[{\bf p}] = \frac{1}{2\,\mu} \sum\limits_{i\in \mathbb N}\sum\limits_{j \in \mathbb N} |i-j|\,p_i\,p_j
\end{equation}
and takes its value in $[0,1]$.
\end{definition}

The Gini index $G$ is a frequently encountered inequality indicator which measures the inequality of a wealth distribution and ranges from $0$
(for a wealth-egalitarian society) to $1$ (as an ever-vanishing proportion of the population holds all the wealth, leading to extreme inequality). We claim that the Gini index serves as a Lyapunov functional, after some finite (and explicitly computable) time, for the evolution system \eqref{eq:law_limit_repeat}-\eqref{eq:Q_rewrite} when $\mu \in (0,1]$. Due to the dependence of the expression for $p_0(t)$ on $\mu$, we split the discussion into two separate cases for ease of presentation. We first investigate the large time behavior of solutions of \eqref{eq:law_limit_repeat}-\eqref{eq:Q_rewrite} when $\mu = 1$, whence the system \eqref{eq:law_limit_repeat}-\eqref{eq:Q_rewrite} reads as
\begin{equation}
  \label{eq:law_limit_mu=1}
  p'_n(t)  = \left\{
    \begin{array}{ll}
      -p^2_0 & \quad n=0, \\
      n\,p_{n+1}+p_0\,p_{n-1}- (n-1)\,p_n - p_0\,p_n, & \quad n \geq 1.
    \end{array}
  \right.
\end{equation}
For notational simplicity, we also write $F_n \coloneqq \sum_{i\leq n} p_i$ for all $n \in \mathbb{N}$ with the convention that $F_{-1}=0$. We now prove the following sharp result for the behavior of the Gini index.
\begin{theorem}\label{thm:1}
Assume that ${\bf p}(t)$ is a solution of the system \eqref{eq:law_limit_mu=1} with ${\bf p}(0) \in \mathcal{S}_1$, then there exists some positive constants $C_1$ and $C_2$ such that
\begin{equation}\label{eq:Gini_decay_1}
\frac{C_1}{t} \leq G[{\bf p}(t)] \leq \frac{C_2}{t}
\end{equation}
for all $t\geq t_* \coloneqq \min\{s \geq 0 \mid p_0(s) \leq \frac{1}{6}\}$. In other words, the decay rate of the Gini index $G[{\bf p}(t)]$ to $G[\delta_1] = 0$ is of order $1\slash t$ and this rate is (asymptotically) sharp.
\end{theorem}

\begin{proof}
We start with certain important preliminary observations by fixing ${\bf p} \in \mathcal{S}_1$ and showing that
\begin{equation}\label{eq:obser1}
F_1 \geq 1 - F_0 = 1 - p_0
\end{equation}
and
\begin{equation}\label{eq:obser2}
G[{\bf p}] \leq 3\,F_0 = 3\,p_0.
\end{equation}
Indeed, inequality \eqref{eq:obser1} is a consequence of the following chain of relations:
\[1 = F_1 - p_0 + \sum_{n\geq 2} n\,p_n \geq F_1 - p_0 + 2\,\sum_{n\geq 2} p_n = F_1 - p_0 + 2\,(1-F_1) = 2 - p_0 - F_1. \]
Equivalently, $p_1 \geq 1 - 2\,F_0 = 1 - 2\,p_0$. The derivation of \eqref{eq:obser2} follows from
\begin{equation*}
\begin{aligned}
G[{\bf p}] &= \frac 12 \sum\limits_{i\geq 0}\sum\limits_{j\geq 0} |i-j|\,p_i\,p_j = \sum\limits_{i\geq 0} p_i\,\sum\limits_{j\leq i} p_j \\
&= \sum\limits_{i\geq 0} i\,p_i\,F_i - \sum\limits_{i\geq 0} p_i\,\sum\limits_{j\leq i} j\,p_j = \sum\limits_{i\geq 0} i\,p_i\,F_i - \sum\limits_{j\geq 0} j\,p_j\,\sum\limits_{i\geq j} p_i \\
&\leq \sum\limits_{i\geq 0} i\,p_i - p_1\,(1-F_0) \leq 1-p_1\,(1-F_0) \\
&\leq 1 - (1-2\,F_0)\,(1-F_0) = 3\,F_0 - 2\,F^2_0 \leq 3\,F_0.
\end{aligned}
\end{equation*}
We now compute the evolution of $G[{\bf p}(t)]$ along the dynamics prescribed by the system \eqref{eq:law_limit_mu=1}, leading us to
\begin{align*}
\frac{\dd}{\dd t} G[{\bf p}] &= \sum\limits_{i\geq 0}\sum\limits_{j\geq 0}|i-j|\,p'_i\,p_j \\
&= \sum\limits_{i\geq 0}\sum\limits_{j\geq 0} |i-j|\,\left(i\,p_{i+1}+p_0\,p_{i-1}\,\mathbbm{1}\{i\geq 1\}- (i-1)\,p_i\,\mathbbm{1}\{i\geq 1\} - p_0\,p_i\right)\,p_j \\
&= \sum\limits_{i\geq 0}\sum\limits_{j\geq 0} |i-j|\,i\,p_{i+1}\,p_j + p_0\,\sum\limits_{i\geq 1}\sum\limits_{j\geq 0} |i-j|\,p_{i-1}\,p_j \\
&\qquad - \sum\limits_{i\geq 1}\sum\limits_{j\geq 0} |i-j|\,(i-1)\,p_i\,p_j - p_0\,\sum\limits_{i\geq 0}\sum\limits_{j\geq 0} |i-j|\,p_i\,p_j \\
&= \sum\limits_{i\geq 0}\sum\limits_{j\geq 0} (\underbrace{p_0\,p_i - i\,p_{i+1}}_{\coloneqq A_i})\,(|i+1-j|-|i-j|)\,p_j.
\end{align*}
Note that $\sum_{i\geq 0} A_i = p_0 - \sum_{i\geq 0} i\,p_{i+1} = 0$, we deduce that
\begin{align*}
\frac{\dd}{\dd t} G[{\bf p}] &= \sum\limits_{i\geq 0}\left(\sum\limits_{j\leq i} A_i\,p_j - \sum\limits_{j > i} A_i\,p_j \right) \\
&= \sum\limits_{i\geq 0} A_i\,\left(\sum\limits_{j\leq i} p_j - \sum\limits_{j > i} p_j\right) = 2\,\sum\limits_{i\geq 0} (p_0\,p_i - i\,p_{i+1})\,\sum\limits_{j\leq i} p_j \\
&= 2\,\sum\limits_{i\geq 0} \left(F_0\,(F_i-F_{i-1})-i\,(F_{i+1}-F_i)\right)\,F_i.
\end{align*}
Thanks to Abel's summation by parts formula, we have \[\sum\limits_{i\geq 0} (F_i-F_{i-1})\,F_i = F^2_0 + \sum\limits_{i\geq 1} (F_i-F_{i-1})\,F_i = F^2_0 + 1 - F^2_0 - \sum\limits_{i\geq 0} (F_{i+1}-F_i)\,F_i.\] Therefore, we obtain
\begin{equation}\label{eq:Gini_derivative_identity}
\frac 12\,\frac{\dd}{\dd t} G[{\bf p}] = F_0 - \sum\limits_{i\geq 0} (F_{i+1}-F_i)\,F_i\,(F_0+i).
\end{equation}
We observe that \[\sum\limits_{i\geq 0} i\,F_i\,(F_{i+1}-F_i) \geq F_1\,\sum\limits_{i\geq 0} i\,(F_{i+1}-F_i) = F_1\,\sum\limits_{i\geq 0} i\,p_{i+1} = F_1\,F_0 \] and that
\begin{align*}
\sum\limits_{i\geq 0} (F_{i+1}-F_i)\,F_i &\geq F_0\,(F_1-F_0)+F_1\,\sum\limits_{i\geq 0} (F_{i+1}-F_i) \\
&= F_0\,(F_1-F_0)+F_1\,(1-F_1) = F_0\,F_1 - F^2_0 + F_1 - F^2_1.
\end{align*}
Thus we deduce from the identity \eqref{eq:Gini_derivative_identity} the following differential inequality:
\begin{equation}\label{eq:Gini_derivative_inequality}
\frac 12\,\frac{\dd}{\dd t} G[{\bf p}] \leq F_0\,\left(1-(2\,F_1-F^2_1+F_0\,F_1-F^2_0)\right) = -F_0\,\left(F_0\,F_1-F^2_0-(1-F_1)^2\right).
\end{equation}
Notice that \[F_0\,F_1-F^2_0-(1-F_1)^2 \geq F_0\,(1-F_0) - F^2_0 - F^2_0 = F_0\,(1-3\,F_0)\] thanks to \eqref{eq:obser1}, from which it follows that \[\frac 12\,\frac{\dd}{\dd t} G[{\bf p}(t)] \leq -\frac{F^2_0(t)}{2}\] for all $t \geq t_* \coloneqq \min\{s \geq 0 \mid p_0(s) \leq \frac{1}{6}\}$. Taking into account of the observation \eqref{eq:obser2}, we arrive at
\begin{equation}\label{eq:Gini_differential_inequality1}
\frac{\dd}{\dd t} G[{\bf p}(t)] \leq -\frac{1}{9}\,G[{\bf p}(t)] \quad \forall~t \geq t_*.
\end{equation}
Thus Gronwall's inequality leads us to
\begin{equation}\label{eq:Gini_upper_bound1}
G[{\bf p}(t)] \leq \frac{9}{t-t_*+\frac{9}{G[{\bf p}(t_*)]}} \quad \forall~t \geq t_*.
\end{equation}
On the other hand, the explicit expression \eqref{eq:p_0;case2} for $p_0(t)$ guarantees the existence of some constant $C > 0$ such that
\begin{equation}\label{eq:b1}
\frac 12\,\frac{\dd}{\dd t} G[{\bf p}(t)] \leq -\frac{F^2_0(t)}{2} \leq -\frac{C}{t^2} \quad \forall~t \geq t_*.
\end{equation}
Upon integrating the relation \eqref{eq:b1} we obtain
\begin{equation}\label{eq:Gini_lower_bound1}
G[{\bf p}(t)] \geq -\int_t^{\infty} \frac{\dd}{\dd s} G[{\bf p}(s)]\,\dd s \geq \int_t^{\infty} \frac{C}{s^2}\,\dd s = \frac{C}{t} \quad \forall~t \geq t_*.
\end{equation}
Combining the bounds \eqref{eq:Gini_upper_bound1} and \eqref{eq:Gini_lower_bound1} yields the claimed result \eqref{eq:Gini_decay_1}.
\end{proof}

To illustrate our quantitative convergence guarantee \eqref{eq:Gini_decay_1} of the Gini index we plot the evolution of the Gini index $G[{\bf p}(t)]$ over time (see figure \ref{fig:cv_gini_mu=1}) when $\mu = 1$, using the same set-up (i.e., model parameters) as before in figure \ref{fig:conv_to_equilibrium}. Notice that the decay of $G[{\bf p}(t)]$ behaves like $\mathcal{O}(1/t)$ when time is moderately large.

\begin{figure}[!htb]
\centering
\includegraphics[width=.97\textwidth]{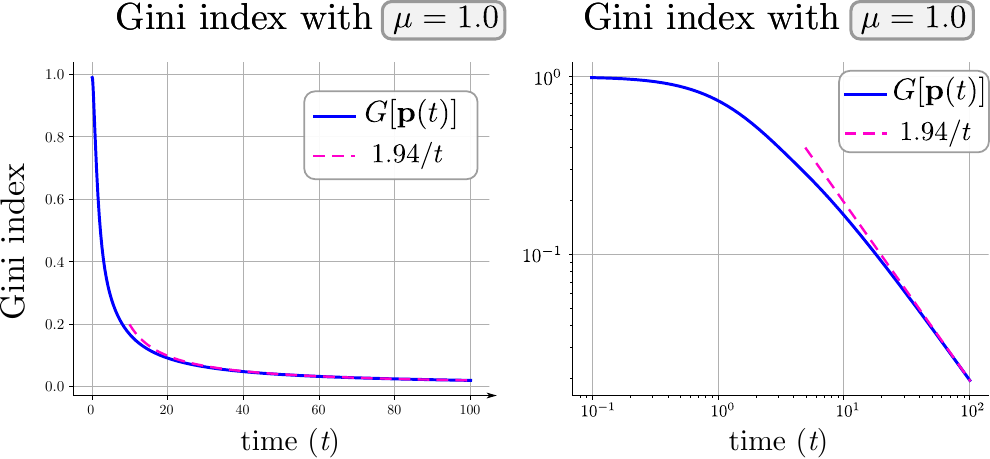}
\caption{{\bf Left}: Decay of the Gini index $G[{\bf p}(t)]$ with respect to time. {\bf Right}: Decay of the Gini index $G[{\bf p}(t)]$ in log-log scale. The decay is inversely proportional to time as predicted by the Theorem \ref{thm:1}.}
\label{fig:cv_gini_mu=1}
\end{figure}

\begin{remark}
In the case where $p_0(0)=1$, we readily deduce from \eqref{eq:obser1} that $p_1(0) \geq 1 - 2\,p_0(0) = 1$, whence $p_1(0) = 1$ and $p_n(0) = 0$ for all $n\geq 2$. Consequently, ${\bf p}(0)$ already coincides with the equilibrium Dirac delta distribution $\delta_1$ \eqref{eq:Dirac_delta}.
\end{remark}

\begin{remark}
It is worth emphasizing that we can not hope for a stronger monotonicity property of the Gini index for all time, i.e., $\frac{\dd}{\dd t} G[{\bf p}(t)] \leq 0$ for all $t \geq 0$. Indeed, if we take ${\bf p}(0) \in \mathcal{S}_1$ such that $p_0 = p_2 = 1$ and $p_n = 0$ for all $n \in \mathbb N \setminus \{0,2\}$, then the identity \eqref{eq:Gini_derivative_identity} for the time derivative of the Gini index along solutions of \eqref{eq:law_limit_mu=1} yields that
\begin{equation}\label{eq:example1}
\frac 12\,\frac{\dd}{\dd t} G[{\bf p}]\Big|_{t=0} = \frac 12 - \frac 12\cdot \frac 12\cdot \frac 32 = \frac 18 > 0.
\end{equation}
\end{remark}

Next, we turn our attention to the large time behavior of solutions of \eqref{eq:law_limit_repeat}-\eqref{eq:Q_rewrite} when $\mu \in (0,1)$. We will prove that the Gini index \eqref{def1:Gini} still serves as an appropriate Lyapunov functional. Before we state the main result, we collect several preliminary observations.

\begin{lemma}\label{lem:observations}
Assume that $\mu \in (0,1)$ and ${\bf p} \in \mathcal{S}_\mu$ \eqref{eq:S_mu}. Then we have
\begin{equation}\label{eq:obser3}
F_1 \geq 2-\mu - p_0,
\end{equation}
and \begin{equation}\label{eq:obser4}
0 \leq G[{\bf p}] - (1-\mu) = G[{\bf p}] - G[{\bf p}^*] \leq 3\,(\mu-1+F_0).
\end{equation}
We remark here that \eqref{eq:obser3} and \eqref{eq:obser4} generalize the previous observations \eqref{eq:obser1} and \eqref{eq:obser2}, respectively.
\end{lemma}

\begin{proof}
First, we note that $p_0 \geq 1-\mu$ and the Bernoulli distribution ${\bf p}^*$ \eqref{eq:Bernoulli} has its Gini index equals to $(1-\mu)$ (i.e., $G[{\bf p}^*] = 1-\mu$). As
\[\mu = F_1 - p_0 + \sum_{n\geq 2} n\,p_n \geq F_1 - p_0 + 2\,\sum_{n\geq 2} p_n = F_1 - p_0 + 2\,(1-F_1) = 2 - p_0 - F_1,\]
it follows that $F_1 \geq 2 - \mu - p_0$ or equivalently $p_1 + 2\,p_0 \geq 2-\mu$. The inequality $G[{\bf p}] - (1-\mu) \geq 0$ is a consequence of the following observation:
\begin{equation*}
G[{\bf p}] = \frac{1}{2\,\mu} \sum\limits_{i\in \mathbb N}\sum\limits_{j \in \mathbb N} |i-j|\,p_i\,p_j \geq \frac{1}{\mu}\,\sum\limits_{j \in \mathbb N} p_0\,j\,p_j = p_0\geq 1-\mu.
\end{equation*}
Next, we recall from the proof of Theorem \ref{thm:1} that \[G[{\bf p}] = \frac{1}{\mu}\,\sum\limits_{i\geq 0} i\,p_i\,F_i - \sum\limits_{j\geq 0} j\,p_j\,\sum\limits_{i\geq j} p_i.\] Since \[\sum_{i\geq 0} i\,p_i\,F_i \leq \sum\limits_{i\geq 0} i\,p_i = \mu \quad \text{and} \quad \sum\limits_{j\geq 0} j\,p_j\,\sum\limits_{i\geq j} p_i \geq p_1\,(1-F_0),\] it follows that
\begin{equation}\label{eq:e1}
G[{\bf p}] \leq \frac{1}{\mu}\,\left(\mu-p_1\,(1-F_0)\right)
\end{equation}
Thanks to the bound \eqref{eq:obser3}, we also obtain
\begin{equation}\label{eq:e2}
1-\frac{p_1}{\mu} \leq 1 - \frac{2-\mu-2\,p_0}{\mu} = 1-\frac{2\,(1-F_0)-\mu}{\mu} = \frac{2\,(\mu-1+F_0)}{\mu}.
\end{equation}
Assembling \eqref{eq:e1} and \eqref{eq:e2} yields the advertised upper bound for $G[{\bf p}] - (1-\mu)$, because
\begin{align*}
G[{\bf p}] - (1-\mu) &\leq \mu - \frac{p_1}{\mu}\,(1-F_0) \\
&= \mu-1+F_0 + (1-F_0)\,\left(1-\frac{p_1}{\mu}\right) \\
&\leq \mu-1+F_0 + 2\,\frac{1-F_0}{\mu}\,(\mu-1+F_0) \leq 3\,(\mu-1+F_0).
\end{align*}
The proof of Lemma \ref{lem:observations} is therefore completed.
\end{proof}

We are now in a position to prove the following main result.
\begin{theorem}\label{thm:2}
Assume that ${\bf p}(t)$ is a solution of the system \eqref{eq:law_limit_repeat}-\eqref{eq:Q_rewrite} with ${\bf p}(0) \in \mathcal{S}_\mu$ and $\mu \in (0,1)$, then for all $t\geq t^* \coloneqq \min\{s \geq 0 \mid p_0(s) \leq 1-\mu+\delta\}$ we have
\begin{equation}\label{eq:Gini_decay_2}
G[{\bf p}(t)] - (1-\mu) \leq \left(G[{\bf p}(t^*)] - (1-\mu)\right)\,\expo^{-\frac{2\,C_\mu}{3\,\mu}\,(t-t^*)}
\end{equation}
where $C_\mu \coloneqq \frac{\min\{\mu,(1-\mu)\}}{4}$, $\delta = K\,\min\{\mu,(1-\mu)\}$ and $K > 0$ is such that $K^2 + 2\,K \leq \frac 12$. On the other hand, we also have
\begin{equation}\label{eq:Gini_lower_bound2}
G[{\bf p}(t)] - (1-\mu) \geq  \frac{\gamma_\mu}{1-\mu}\,\expo^{-(1-\mu)\,t} \quad \forall~t \geq t^*,
\end{equation}
where $\gamma_\mu \coloneqq \frac{2}{\mu}\,C_\mu\,(1-\mu)\,\frac{\mu-1+p_0(0)}{p_0(0)}$.
\end{theorem}

\begin{proof}
Preceding as in the proof of Theorem \ref{thm:1}, we arrive at the following expression for the time derivative of $G[{\bf p}(t)]$:
\begin{equation}\label{eq:Gini_derivative_identity2}
\frac{\mu}{2}\,\frac{\dd}{\dd t} G[{\bf p}] = \mu-1+F_0 - \sum\limits_{i\geq 0} (F_{i+1}-F_i)\,F_i\,(F_0+i+\mu-1),
\end{equation}
which is a (straightforward) generalization of the identity \eqref{eq:Gini_derivative_identity} shown in Theorem \ref{thm:1}. We notice again that \[\sum\limits_{i\geq 0} i\,F_i\,(F_{i+1}-F_i) \geq F_1\,\sum\limits_{i\geq 0} i\,(F_{i+1}-F_i) = F_1\,\sum\limits_{i\geq 0} i\,p_{i+1} = F_1\,(\mu-1+F_0) \] as well as
\begin{align*}
\sum\limits_{i\geq 0} (F_{i+1}-F_i)\,F_i &\geq F_0\,(F_1-F_0)+F_1\,\sum\limits_{i\geq 0} (F_{i+1}-F_i) \\
&= F_0\,(F_1-F_0)+F_1\,(1-F_1) = F_0\,F_1 - F^2_0 + F_1 - F^2_1.
\end{align*}
As a consequence, we deduce from \eqref{eq:Gini_derivative_identity2} that
\begin{equation}\label{eq:Gini_derivative_inequality2}
\frac{\mu}{2}\,\frac{\dd}{\dd t} G[{\bf p}] \leq = -(\mu-1+F_0)\,\left(F_0\,F_1-F^2_0-(1-F_1)^2\right).
\end{equation}
By \eqref{eq:obser3} from Lemma \ref{lem:observations}, \[F_0\,F_1-F^2_0-(1-F_1)^2 \geq F_0\,(2-\mu-F_0) - F^2_0 - (\mu-1+F_0)^2.\] Now for all $t \geq t^*$,
\begin{equation}\label{eq:chain2}
\begin{aligned}
F_0\,(2-\mu-F_0) - F^2_0 - (\mu-1+F_0)^2 &\geq F_0\,(1-\delta)-F^2_0-\delta^2 = F_0\,(1-\delta-F_0)-\delta^2 \\
&\geq F_0\,(\mu-2\,\delta)-\delta^2 \geq (1-\mu)\,(\mu-2\,\delta)-\delta^2 \\
&\geq \frac{\min\{\mu,(1-\mu)\}}{4} = C_\mu.
\end{aligned}
\end{equation}
We remark here that the last inequality in \eqref{eq:chain2} follows from the following reasoning: if $\mu \leq \frac 12$, then \[(1-\mu)\,(\mu-2\,\delta)-\delta^2 \geq \frac 12\,(\mu-2\,\delta)-\delta^2 = \left(\frac 12 - (K^2+K)\,\mu\right)\,\mu \geq \left(\frac 12 - \frac 12\,\mu\right)\,\mu \geq \frac{\mu}{4}.\] On the other hand, if $\frac 12 \leq \mu < 1$, we have
\begin{align*}
(1-\mu)\,(\mu-2\,\delta)-\delta^2 &= (1-\mu)\,\left(\mu-2\,(1-\mu)\,K-K^2\,(1-\mu)\right) \\
&= (1-\mu)\,\left(\mu-(1-\mu)\,(K^2+2\,K)\right) \\
&\geq (1-\mu)\,\left(\mu - \frac{1}{4}\right) \geq \frac{1-\mu}{4}.
\end{align*}
Inserting the estimate \eqref{eq:chain2} into \eqref{eq:Gini_derivative_inequality2} gives rise to
\begin{equation}\label{eq:Gini_differential_inequality3}
\frac{\mu}{2}\,\frac{\dd}{\dd t} G[{\bf p}(t)] \leq -C_\mu\,\left(\mu-1+F_0(t)\right) \leq -\frac{C_\mu}{3}\,\left(G[{\bf p}(t)]-(1-\mu)\right) \quad \forall~t \geq t^*,
\end{equation}
where the last inequality in \eqref{eq:Gini_differential_inequality3} follows from \eqref{eq:obser4} proved in Lemma \ref{lem:observations}. An routine application of Grownall's inequality yields \[G[{\bf p}(t)] - (1-\mu) \leq \left(G[{\bf p}(t^*)] - (1-\mu)\right)\,\expo^{-\frac{2\,C_\mu}{3\,\mu}\,(t-t^*)}\quad \forall~t \geq t^*.\] To derive the bound reported in \eqref{eq:Gini_lower_bound2}, we invoke the explicit formula for $p_0(t)$ \eqref{eq:p_0;case1} to obtain \[F_0(t) - (1-\mu) \geq (1-\mu)\,\frac{\mu-1+p_0(0)}{p_0(0)}\,\expo^{(\mu-1)\,t},\] whence \eqref{eq:Gini_differential_inequality3} implies that
\begin{equation}\label{eq:b2}
\frac{\dd}{\dd t} G[{\bf p}(t)] \leq -\frac{2}{\mu}\,C_\mu\,\left(\mu-1+F_0(t)\right) \leq -\gamma_\mu\,\expo^{(\mu-1)\,t} \quad \forall~t \geq t^*.
\end{equation}
By integrating the differential inequality \eqref{eq:b2} we get
\begin{equation*}
G[{\bf p}(t)] - (1-\mu) \geq -\int_t^{\infty} \frac{\dd}{\dd s} G[{\bf p}(s)]\,\dd s \geq \frac{\gamma_\mu}{1-\mu}\,\expo^{-(1-\mu)\,t} \quad \forall~t \geq t^*,
\end{equation*}
as desired.
\end{proof}

To illustrate the quantitative convergence result \eqref{eq:Gini_decay_2} of the Gini index we plot the evolution of the Gini index $G[{\bf p}(t)]$ over time (see figure \ref{fig:cv_gini_mu=06}) when $\mu = 0.6$, under the same settings as used for figure \ref{fig:cv_gini_mu=1}. It is readily observed that the decay of $G[{\bf p}(t)]$ is exponential with respect to time (at least for large enough $t$).

\begin{figure}[!htb]
\centering
\includegraphics[width=.97\textwidth]{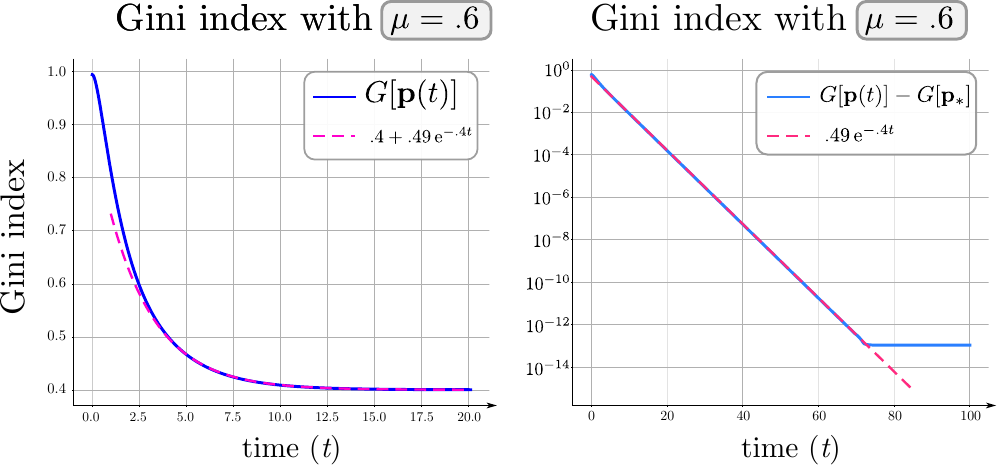}
\caption{{\bf Left}: Decay of the Gini index $G[{\bf p}(t)]$ with respect to time. {\bf Right}: Decay of the Gini index $G[{\bf p}(t)]$ in semilog scale. The decay is exponential with respect to time, as justified by the Theorem \ref{thm:2}.}
\label{fig:cv_gini_mu=06}
\end{figure}

Interestingly, the numerical experiment reported in figure \ref{fig:cv_gini_mu=06} suggests that the upper bound on $G[{\bf p}(t)] - (1-\mu)$ guaranteed by \eqref{eq:Gini_decay_2} might be refined to match the lower bound \eqref{eq:Gini_lower_bound2}, in the sense that it might be possible to prove a reversed version of \eqref{eq:Gini_lower_bound2} (at the price of a different constant in front of the exponential factor $\expo^{-(1-\mu)\,t}$). Given the fact that we are able to prove a two-sided sharp bound \eqref{eq:Gini_decay_1} on the behavior of the Gini index in the scenario where $\mu = 1$, it is quite natural to speculate that a similar two-sided sharp bound might be established for $\mu \in (0,1)$ as well, leading us to
\begin{equation}\label{eq:Gini_conjecture}
C_1\,\expo^{-(1-\mu)\,t} \leq G[{\bf p}(t)] - (1-\mu) \leq  C_2\,\expo^{-(1-\mu)\,t}
\end{equation}
for all large $t$, in which $C_1$ and $C_2$ are two absolute constants depending only on $\mu$. At this point, we leave the potential refinement of the rate of exponential decay reported in \eqref{eq:Gini_decay_2} as a future work.

\begin{remark}
Similar to a previous remark, it is not possible to establish a stronger result such as $\frac{\dd}{\dd t} G[{\bf p}(t)] \leq 0$ for all $t \geq 0$. For instance, if we take $\mu = \frac 12$ and ${\bf p}(0) \in \mathcal{S}_{\frac 12}$ such that $p_0 = \frac 34$, $p_2 = \frac 14$, and $p_n = 0$ for all $n \in \mathbb N \setminus \{0,2\}$, then the identity \eqref{eq:Gini_derivative_identity2} for the time derivative of the Gini index along solutions of \eqref{eq:law_limit_repeat}-\eqref{eq:Q_rewrite} allows us to arrive at
\begin{equation}\label{eq:example2}
\frac 14\,\frac{\dd}{\dd t} G[{\bf p}]\Big|_{t=0} = \frac 12 - 1 + \frac 34 - \frac 14\cdot \frac 34\cdot \frac 54 = \frac{1}{64} > 0.
\end{equation}
\end{remark}

\begin{remark}
As a consequence of Theorem \ref{thm:2}, the decay rate of $G[{\bf p}(t)]$ to $G[{\bf p}^*] = 1-\mu$ is at most $\mathcal{O}\left(\expo^{-(1-\mu)\,t}\right)$. Meanwhile, the aforementioned decay rate is at least $\Omega\left(\expo^{-\frac 16\,t}\right)$ for $\mu \leq \frac 12$ and is at least $\Omega\left(\expo^{-\frac{1-\mu}{6\,\mu}\,t}\right)$ when $\frac 12 \leq \mu < 1$.
\end{remark}

\begin{remark}
At this point, it is plausible to believe that the reinterpretation of the sticky dispersion model (on the complete graph) using terminologies from econophysics is quite beneficial, in the sense that it makes the Gini index \eqref{def1:Gini} as a natural candidate Lyapunov functional for the examination of the large time behavior of the mean-field ODE system \eqref{eq:law_limit_repeat}-\eqref{eq:Q_rewrite}, at least when $\mu \leq 1$.
\end{remark}

\subsection{Relaxation to modified Poisson distribution for $\mu > 1$}\label{subsec:sec3.2}

Our goal in this section is to analyze the asymptotic behaviour of the system \eqref{eq:law_limit_repeat}-\eqref{eq:Q_rewrite} when the parameter $\mu \in (1,\infty)$, and establish convergence of solutions ${\bf p}(t)$ of \eqref{eq:law_limit_repeat}-\eqref{eq:Q_rewrite} to the modified Poisson distribution $\overbar{\bf p}$ \eqref{eq:modified_Poisson} as $t \to \infty$. For this purpose, we point out the following key observation regarding the operator $Q$, whose proof consists of elementary algebraic manipulations and will be omitted.

\begin{lemma}\label{lem:decomposition_of_Q}
For each ${\bf p} \in \mathcal{P}(\mathbb N)$, we can decompose the operator $Q$ \eqref{eq:Q_rewrite} as
\begin{equation}\label{eq:Q_decompose}
Q[{\bf p}] = \widehat{Q}[{\bf p}] - p_0\,\mathcal{D}^-[{\bf p}],
\end{equation}
where the operators $\hat{Q}$ and $\mathcal{D}^-$ are defined, respectively, by
\begin{equation}\label{eq:Q_hat}
  \widehat{Q}[{\bf p}]_n = \left\{
    \begin{array}{ll}
      -(\mu-1)\,p_0 & \quad n=0, \\
      n\,p_{n+1}+(\mu-1)\,p_{n-1}- (n-1)\,p_n - (\mu-1)\,p_n, & \quad n \geq 1.
    \end{array}
  \right.
\end{equation}
and $\mathcal{D}^-[{\bf p}] = {\bf p} - \mathcal{R}[{\bf p}]$ with
\begin{equation}\label{eq:R}
  \mathcal{R}[{\bf p}]_n = \left\{
    \begin{array}{ll}
     0 & \quad n=0, \\
     p_{n-1}, & \quad n \geq 1.
    \end{array}
  \right.
\end{equation}
In particular, the operator $\widehat{Q}$ is \emph{linear} and preserves the total probability mass (i.e., $\widehat{Q}[{\bf p}] \in \mathcal{P}(\mathbb N)$).
\end{lemma}

It is also straightforward to check that $\overbar{\bf p} \in \ker(\widehat{Q})$, i.e., $\widehat{Q}[\overbar{\bf p}] = {\bf 0}$. As a consequence of Lemma
\eqref{lem:decomposition_of_Q}, we can recast the ODE system \eqref{eq:law_limit_repeat}-\eqref{eq:Q_rewrite} as
\begin{equation}\label{eq:main_ODE}
\frac{\dd}{\dd t} {\bf p}(t) = \widehat{Q}[{\bf p}(t)] - p_0(t)\,\mathcal{D}^-[{\bf p}(t)].
\end{equation}
As we have an explicit expression for $p_0(t)$ \eqref{eq:p_0;case1} (depending only on $p_0(0)$), the system \eqref{eq:main_ODE} is ``almost linear''. In fact, once the initial data ${\bf p}(0) \in \mathcal{S}_\mu$ (or $p_0(0)$ to be more precise) is provided/known, the evolution system \eqref{eq:main_ODE} will become a linear (but non-homogeneous in time) dynamical system. Motivated by these observations, we decide to view \eqref{eq:main_ODE} as a perturbed version of the following linear ODE system

\begin{equation}\label{eq:linear_ODE_hat}
\frac{\dd}{\dd t} \widehat{{\bf p}}(t) = \widehat{Q}[\widehat{{\bf p}}(t)].
\end{equation}

We now study the large time behaviour of the linear equation \eqref{eq:linear_ODE_hat} as it will facilitate our investigation of the original system \eqref{eq:main_ODE}. It turns out the linear system admits a Fokker-Planck formulation, similar to the so-called poor-biased exchange model introduced and studied in the recent work \cite{cao_derivation_2021}, whence techniques employed in \cite{cao_derivation_2021} can be applied here with suitable modifications. We start with several preliminary definitions.

\begin{definition}\label{def:2}
For any non-negative sequence $\{a_n\}_{n \in \mathbb N}$, we define $D^+ a_n \coloneqq a_{n+1}-a_n$ for all $n\geq 0$, $D^- a_n \coloneqq a_n - a_{n-1}$ for all $n \geq 1$, and $D^- a_0 \coloneqq a_0$. We will also adopt the convention that $\frac{a}{a} = 1$ when $a = 0$. Lastly, we define $\Delta a_n = D^-(D^+ a_n)$ for all $n\geq 0$.
\end{definition}

\begin{definition}\label{def:3}
We define the following sub-vector space of $\ell^2(\mathbb N)$:
\begin{equation}\label{eq:H0}
\mathcal{H} \coloneqq \left\{{\bf p} \in \ell^2(\mathbb{N}) \mid p^2_0 + \sum_{n=1}^\infty \frac{p^2_n}{\overbar{p}_n} < +\infty\right\}
\end{equation}
along with the induced scalar product
\begin{equation}\label{eq:scalar_prod_H}
\langle{\bf p},\mathbf{q} \rangle_{\mathcal{H}}  \coloneqq p_0\,q_0 + \sum_{n=1}^\infty \frac{p_n\,q_n}{\overbar{p}_n}.
\end{equation}
\end{definition}

Armed with these definitions, we then identity the Fokker-Planck structure behind the linear system \eqref{eq:linear_ODE_hat}.

\begin{lemma}\label{lem:FK_Qhat}
For each ${\bf p} \in \mathcal H$, we have
\begin{equation}\label{eq:FK_Qhat}
\widehat{Q}[{\bf p}]_n =  (\mu-1)\,D^-\left(\overbar{p}_n\,D^+ \left(\frac{p_n}{\overbar{p}_n}\right)\right)
\end{equation}
for all $n \in \mathbb N$.
\end{lemma}

\begin{proof}
For $n=0$, we have \[(\mu-1)\,D^-\left(\overbar{p}_n\,D^+ \left(\frac{p_n}{\overbar{p}_n}\right)\right) = (\mu-1)\,\overbar{p}_0\,D^+ \left(\frac{p_0}{\overbar{p}_0}\right) = (\mu-1)\,\overbar{p}_0\,\left(\frac{p_1}{\overbar{p}_1}-\frac{p_0}{\overbar{p}_0}\right) = -(\mu-1)\,p_0.\] When $n\geq 1$, we compute
\begin{align*}
\frac{1}{\mu-1}\,\widehat{Q}[{\bf p}]_n  &= \frac{\overbar{p}_n}{\overbar{p}_{n+1}}\,p_{n+1}-\frac{\overbar{p}_{n-1}}{\overbar{p}_{n}}\,p_{n} - \left(\frac{\overbar{p}_n}{\overbar{p}_{n}}\,p_{n}-\frac{\overbar{p}_{n-1}}{\overbar{p}_{n-1}}\,p_{n-1}\right) \\
&= D^-\left(\overbar{p}_n\,D^+ \left(\frac{p_n}{\overbar{p}_n}\right)\right).
\end{align*}
Thus the proof of Lemma \ref{lem:FK_Qhat} is completed.
\end{proof}

As an immediate corollary of Lemma \ref{lem:FK_Qhat}, we deduce the following integration by parts formula:
\begin{corollary}\label{Cor:1}
Assume that $\widehat{{\bf p}}(t)$ is a solution of the linear ODE system \eqref{eq:linear_ODE_hat}, then for any ${\bf \varphi} \in \mathcal H$ we have
\begin{equation}\label{eq:IBP_linear_ODE}
\sum\limits_{n=1}^\infty \widehat{p}'_n\,\varphi_n = -(\mu-1)\,\sum\limits_{n=1}^\infty \overbar{p}_n\,D^+\left(\frac{\widehat{p}_n}{\overbar{p}_n}\right)\,D^+\varphi_n + (\mu-1)\,\widehat{p}_0\,\varphi_1.
\end{equation}
In particular, taking $\varphi_n = \frac{\widehat{p}_n}{\overbar{p}_n}$ for all $n\geq 1$ gives rise to
\begin{equation}\label{eq:IBP_linear_corollary}
\frac 12\,\frac{\dd}{\dd t} \sum\limits_{n=1}^\infty \frac{(\widehat{p}_n-\overbar{p}_n)^2}{\overbar{p}_n}  = -(\mu-1)\,\sum\limits_{n=1}^\infty \overbar{p}_n\,\left|D^+\left(\frac{\widehat{p}_n}{\overbar{p}_n}\right)\right|^2 + (\mu-1)\,\widehat{p}_0\,\frac{\widehat{p}_1}{\overbar{p}_1} - (\mu-1)\,\widehat{p}_0.
\end{equation}
\end{corollary}

\begin{proof}
We skip the derivation of the basic formula \eqref{eq:IBP_linear_ODE} and detail the proof of \eqref{eq:IBP_linear_corollary}. First, notice that
\[\sum\limits_{n=1}^\infty \frac{(\widehat{p}_n-\overbar{p}_n)^2}{\overbar{p}_n} = \sum\limits_{n=1}^\infty \frac{\widehat{p}^2_n}{\overbar{p}_n} -2\,\sum\limits_{n=1}^\infty \widehat{p}_n + \sum\limits_{n=1}^\infty \overbar{p}_n = \sum\limits_{n=1}^\infty \frac{\widehat{p}^2_n}{\overbar{p}_n} - 2\,(1-\widehat{p}_0) + 1 = \sum\limits_{n=1}^\infty \frac{\widehat{p}^2_n}{\overbar{p}_n} + 2\,\widehat{p}_0 - 1.\] Invoking \eqref{eq:IBP_linear_ODE} with $\varphi_n = \frac{\widehat{p}_n}{\overbar{p}_n}$ for all $n\geq 1$ yields that
\begin{align*}
\frac 12\,\frac{\dd}{\dd t} \sum\limits_{n=1}^\infty \frac{(\widehat{p}_n-\overbar{p}_n)^2}{\overbar{p}_n} &= \frac 12\,\frac{\dd}{\dd t} \sum\limits_{n=1}^\infty \frac{\widehat{p}^2_n}{\overbar{p}_n} + \widehat{p}'_0 \\
&= -(\mu-1)\,\sum\limits_{n=1}^\infty \overbar{p}_n\,\left|D^+\left(\frac{\widehat{p}_n}{\overbar{p}_n}\right)\right|^2 + (\mu-1)\,\widehat{p}_0\,\frac{\widehat{p}_1}{\overbar{p}_1} - (\mu-1)\,\widehat{p}_0,
\end{align*}
completing the proof of \eqref{eq:IBP_linear_corollary}.
\end{proof}

In view of the content of Corollary \ref{Cor:1}, an upper bound for $\sum\limits_{n=1}^\infty \frac{(\widehat{p}_n-\overbar{p}_n)^2}{\overbar{p}_n}$ (at time $t$) can be derived as long as we can control $\sum\limits_{n=1}^\infty \frac{(\widehat{p}_n-\overbar{p}_n)^2}{\overbar{p}_n}$ in terms of $\sum\limits_{n=1}^\infty \overbar{p}_n\,\left|D^+\left(\frac{\widehat{p}_n}{\overbar{p}_n}\right)\right|^2$. Inspired from the recent work \cite{cao_derivation_2021} on a related model, we prove a (weak) Poincare-type inequality for the modified Poisson distribution $\overbar{{\bf p}}$ \eqref{eq:modified_Poisson}, whose proof will be used on the classical Poisson-Poincare inequality stated in \cite{boucheron_concentration_2013}.

\begin{proposition}\label{prop:Weak_Poincare}
For each ${\bf p} \in \mathcal{P}(\mathbb N)$ we have
\begin{equation}\label{eq:weak_Poincare}
\sum\limits_{n=1}^\infty \frac{(p_n-\overbar{p}_n)^2}{\overbar{p}_n} \leq p^2_0 + (\mu-1)\,\sum\limits_{n=1}^\infty \overbar{p}_n\,\left|D^+\left(\frac{p_n}{\overbar{p}_n}\right)\right|^2.
\end{equation}
\end{proposition}

\begin{proof}
We first recall the standard Poisson-Poincare inequality provided in the monograph \cite{boucheron_concentration_2013} and revisited in \cite{cao_derivation_2021}. Indeed, the routine Poisson-Poincare inequality states that
\begin{equation}\label{eq:Poisson-Poincare}
\Var\left[f(Y)\right] \leq \nu\,\mathbb{E}\left[\left(f(Y+1)-f(Y)\right)^2\right]
\end{equation}
for any Poisson distributed random variable $Y$ with mean value $\nu$, where $f$ is a real-valued function. Thanks to the probabilistic interpretation of the modified Poisson distribution \eqref{eq:modified_Poisson} provided in a remark before section \ref{subsec:sec3.2}, we know that if we take $Y \sim \textrm{Possion}(\mu-1)$, then the law of $X \coloneqq 1 + Y$ coincides with $\overbar{{\bf p}}$. Therefore, by shifting the (test) function $f$ and define $\tilde{f}(x) \coloneqq f(x-1)$, we deduce from \eqref{eq:Poisson-Poincare} that
\begin{equation}\label{eq:Poisson-Poincare2}
\Var\left[\tilde{f}(X)\right] \leq (\mu-1)\,\mathbb{E}\left[\left(\tilde{f}(X+1)-\tilde{f}(X)\right)^2\right].
\end{equation}
Therefore, setting $\tilde{f}_n = \frac{p_n}{\overbar{p}_n}$ for $n\geq 1$ and $\tilde{f}_0 = 0$ into \eqref{eq:Poisson-Poincare2}, we obtain
\begin{align*}
\Var\left[\tilde{f}(X)\right] &= \mathbb{E}\left[\tilde{f}(X)-(1-p_0)\right]^2 = \sum\limits_{n=1}^\infty \left(\frac{p_n}{\overbar{p}_n}-1+p_0\right)^2\,\overbar{p}_n \\
&= \sum\limits_{n=1}^\infty \frac{(p_n-\overbar{p}_n)^2}{\overbar{p}_n} - p^2_0 \\
&\leq (\mu-1)\,\sum\limits_{n=1}^\infty \overbar{p}_n\,\left|D^+\left(\frac{p_n}{\overbar{p}_n}\right)\right|^2,
\end{align*}
which finishes the proof.
\end{proof}

Now we have all the required pieces to study the large time behaviour of solutions to the linear system \eqref{eq:linear_ODE_hat}. For the remaining part of this section, we will use $C$ to represent a generic positive (and bounded) constant whose value may vary from line to line.

\begin{proposition}\label{prop:linear_ODE} 
Suppose that $\widehat{{\bf p}}(t)$ is a solution of the linear system \eqref{eq:linear_ODE_hat} with initial condition $\widehat{{\bf p}}(0) \in \mathcal{P}(\mathbb N)$. Let $\HH(t) \coloneqq \frac 12\,\sum\limits_{n=1}^\infty \frac{(\widehat{p}_n(t)-\overbar{p}_n)^2}{\overbar{p}_n}$. Then for all $t \geq 0$,
\begin{equation}\label{eq:exponential_decay_Qhat}
\HH(t) \leq \left\{\begin{aligned}
&\HH(0)\,\expo^{-2\,t} + \frac{K_\mu}{\mu-3}\,\left(\expo^{-2\,t} - \expo^{-(\mu-1)\,t}\right) \quad \text{if~} \mu \neq 3,\\
&\HH(0)\,\expo^{-2\,t} + K_3\,t\,\expo^{-2\,t} \quad \textrm{if~} \mu = 3,\\
\end{aligned}
\right.
\end{equation}
in which $K_\mu \coloneqq \left(1 + (\mu-1)\,\expo^{\mu-1}\right)\,\widehat{p}_0(0)$. Consequently, for all $t \geq 0$ we also have
\begin{equation}\label{eq:exponential_decay_Hnorm}
\|\widehat{{\bf p}}(t) - \overbar{{\bf p}}\|^2_{\mathcal H} \leq \left\{\begin{aligned}
&\widehat{p}^2_0(0)\,\expo^{-2\,(\mu-1)\,t} +  2\,\HH(0)\,\expo^{-2\,t} + \frac{2\,K_\mu}{\mu-3}\,\left(\expo^{-2\,t} - \expo^{-(\mu-1)\,t}\right)\quad \textrm{if~} \mu \neq 3,\\
&\widehat{p}^2_0(0)\,\expo^{-4\,t} +  2\,\HH(0)\,\expo^{-2\,t} + 2\,K_3\,t\,\expo^{-2\,t} \quad \text{if~} \mu = 3.
\end{aligned}
\right.
\end{equation}
In other words,
\begin{equation}\label{eq:exponential_decay_Hnorm_compact_form}
\|\widehat{{\bf p}}(t) - \overbar{{\bf p}}\|^2_{\mathcal H} \leq \left\{\begin{aligned}
&C\,\expo^{-(\mu-1)\,t} \quad \textrm{if~} 1 < \mu < 3,\\
&C\,t\,\expo^{-2\,t} \quad \text{if~} \mu = 3, \\
&C\,\expo^{-2\,t} \quad \text{if~} \mu > 3.
\end{aligned}
\right.
\end{equation}
\end{proposition}

\begin{proof}
We only treat the case when $\mu \in (1,3) \cup (3,\infty)$ since the other case can be handled using a similar argument. Starting from the identity \eqref{eq:IBP_linear_corollary}, we employ the conclusion of Proposition \ref{prop:Weak_Poincare} and the analytical expression for $\widehat{p}_0(t)$ to deduce that
\begin{equation}\label{eq:H_derivative}
\begin{aligned}
\frac{\dd}{\dd t} \HH &= -(\mu-1)\,\sum\limits_{n=1}^\infty \overbar{p}_n\,\left|D^+\left(\frac{\widehat{p}_n}{\overbar{p}_n}\right)\right|^2 + (\mu-1)\,\widehat{p}_0\,\frac{\widehat{p}_1}{\overbar{p}_1} - (\mu-1)\,\widehat{p}_0 \\
&\leq -2\,\HH + \widehat{p}^2_0 + (\mu-1)\,\widehat{p}_0\,\frac{\widehat{p}_1}{\overbar{p}_1} - (\mu-1)\,\widehat{p}_0 \\
&\leq -2\,\HH + \widehat{p}_0 + (\mu-1)\,\expo^{\mu-1}\,\widehat{p}_0 \\
&= -2\,\HH + K_\mu\,\expo^{-(\mu-1)\,t}.
\end{aligned}
\end{equation}
As a consequence of Grownall's inequality, we end up with
\begin{equation*}
\begin{aligned}
\HH(t) &\leq \HH(0)\,\expo^{-2\,t} + K_\mu\,\int_0^t \expo^{-2\,(t-s)}\,\expo^{-(\mu-1)\,s}\,\dd s \\
&= \HH(0)\,\expo^{-2\,t} + K_\mu\,\expo^{-2\,t}\,\int_0^t \expo^{-(\mu-3)\,s}\,\dd s \\
&= \HH(0)\,\expo^{-2\,t} + \frac{K_\mu}{\mu-3}\,\left(\expo^{-2\,t} - \expo^{-(\mu-1)\,t}\right).
\end{aligned}
\end{equation*}
Since \[\|\widehat{{\bf p}}(t) - \overbar{{\bf p}}\|^2_{\mathcal H} = \widehat{p}^2_0(t) + \sum_{n=1}^\infty \frac{(\widehat{p}_n(t)-\overbar{p}_n)^2}{\overbar{p}_n} = \widehat{p}^2_0(0)\,\expo^{-2\,(\mu-1)\,t} + 2\,\HH(t),\] the bound \eqref{eq:exponential_decay_Hnorm} follows readily from \eqref{eq:exponential_decay_Qhat}.
\end{proof}

Next, we prove our main result in this section, regarding the large time asymptotic of solutions to the original system \eqref{eq:main_ODE}.

\begin{theorem}\label{thm:3}
Assume that ${\bf p}(t)$ is a solution of the system \eqref{eq:main_ODE} with ${\bf p}(0) \in \mathcal{S}_\mu$ and $\mu \in (1,\infty)$, then for all $t \geq 0$ we have
\begin{equation}\label{eq:bound_main}
\|{\bf p}(t) - \overbar{{\bf p}}\|^2_{\mathcal H} \leq \left\{\begin{aligned}
&C\,t\,\expo^{-(\mu-1)\,t} \quad \textrm{if~} 1 < \mu < 3,\\
&C\,t^2\,\expo^{-2\,t} \quad \text{if~} \mu = 3, \\
&C\,\expo^{-2\,t} \quad \text{if~} \mu > 3.
\end{aligned}
\right.
\end{equation}
\end{theorem}

\begin{proof}
By Duhamel's principle, we can write the solution to \eqref{eq:main_ODE} as
\begin{equation*}
{\bf p}(t) = \expo^{t\,\widehat{Q}}\,{\bf p}(0) - \int_0^t p_0(s)\,\expo^{(t-s)\,\widehat{Q}}\,\mathcal{D}^-[{\bf p}(s)]\,\dd s.
\end{equation*}
Thus, we can bound $\|{\bf p}(t) - \overbar{{\bf p}}\|^2_{\mathcal H}$ as follows:
\begin{align*}
\|{\bf p}(t) - \overbar{{\bf p}}\|^2_{\mathcal H} &\leq 2\,\|\expo^{t\,\widehat{Q}}\,{\bf p}(0) - \overbar{{\bf p}}\|^2_{\mathcal H} + 2\,\int_0^t p_0(s)\,\|\expo^{(t-s)\,\widehat{Q}}\,\mathcal{D}^-[{\bf p}(s)]\|^2_{\mathcal H}\, \dd s \\
&\leq 2\,\|\expo^{t\,\widehat{Q}}\,{\bf p}(0) - \overbar{{\bf p}}\|^2_{\mathcal H} + 4\,\int_0^t p_0(s)\,\|\expo^{(t-s)\,\widehat{Q}}{\bf p}(s) - \overbar{{\bf p}}\|^2_{\mathcal H}\,\dd s \\
&\qquad  + 4\,\int_0^t p_0(s)\,\|\expo^{(t-s)\,\widehat{Q}}\mathcal{R}[{\bf p}(s)] - \overbar{{\bf p}}\|^2_{\mathcal H}\,\dd s.
\end{align*}
Using the estimates \eqref{eq:exponential_decay_Hnorm_compact_form} from Proposition \ref{prop:linear_ODE} and recall the known formula for $p_0(t)$ \eqref{eq:p_0;case1}, we have, for $\mu \in (1,3)$, that
\begin{align*}
\|{\bf p}(t) - \overbar{{\bf p}}\|^2_{\mathcal H} &\leq C\,\expo^{-(\mu-1)\,t} + C\,\int_0^t \expo^{-(\mu-1)\,s}\,\expo^{-(\mu-1)\,(t-s)}\,\dd s\\
&= C\,\expo^{-(\mu-1)\,t} + C\,\int_0^t \expo^{-(\mu-1)\,t}\,\dd s \\
&\leq C\,t\,\expo^{-(\mu-1)\,t}.
\end{align*}
Similarly, for $\mu = 3$ we have
\begin{align*}
\|{\bf p}(t) - \overbar{{\bf p}}\|^2_{\mathcal H} &\leq C\,t\,\expo^{-2\,t} + C\,\int_0^t \expo^{-2\,s}\,t\,\expo^{-(\mu-1)\,(t-s)}\,\dd s\\
&= C\,\expo^{-(\mu-1)\,t} + C\,\int_0^t t\,\expo^{-2\,t}\,\dd s \\
&\leq C\,t^2\,\expo^{-(\mu-1)\,t}.
\end{align*}
Finally, for $\mu \in (3,\infty)$ we obtain
\begin{align*}
\|{\bf p}(t) - \overbar{{\bf p}}\|^2_{\mathcal H} &\leq C\,\expo^{-2\,t} + C\,\int_0^t \expo^{-(\mu-1)\,s}\,\expo^{-2\,(t-s)}\,\dd s\\
&= C\,\expo^{-2\,t} + C\,\expo^{-2\,t}\,\int_0^t \expo^{-(\mu-3)\,s}\,\dd s \\
&\leq C\,\expo^{-2\,t}.
\end{align*}
Thus the proof of Theorem \ref{thm:3} is completed.
\end{proof}

To illustrate the quantitative convergence of ${\bf p}(t)$ towards the modified Poisson distribution $ \overbar{{\bf p}}$ we plot the evolution of $\|{\bf p}(t) - \overbar{{\bf p}}\|_{\mathcal H}$ and $\|{\bf p}(t) - \overbar{{\bf p}}\|_{\ell^2}$ over time (see figure \ref{fig:cv_energy}) for two different choices of $\mu > 1$, under the same settings as the previous subsection. It is readily observed that the decay of both energies are exponential with respect to time as we employ the semi-logarithmic scale.

\begin{figure}[!htb]
  \begin{subfigure}{0.45\textwidth}
    \centering
    \includegraphics[scale=0.8]{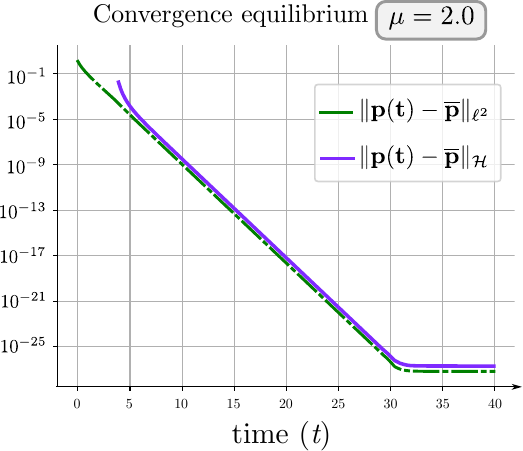}
  \end{subfigure}
  \hspace{0.1in}
  \begin{subfigure}{0.45\textwidth}
    \centering
    \includegraphics[scale=0.8]{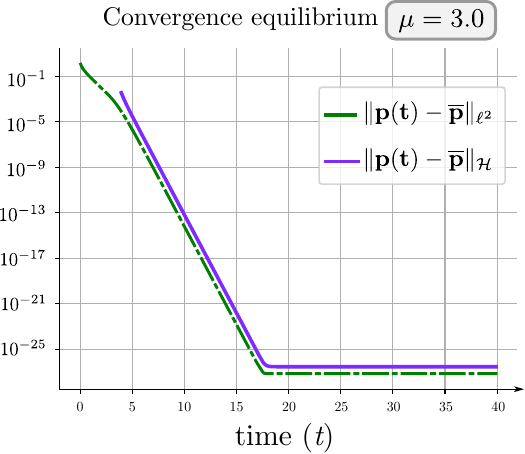}
  \end{subfigure}
  \caption{{\bf Left}: Decay of $\|{\bf p}(t) - \overbar{{\bf p}}\|_{\mathcal H}$ and $\|{\bf p}(t) - \overbar{{\bf p}}\|_{\ell^2}$ in semilog scale when $\mu = 2$. {\bf Right}: Decay of $\|{\bf p}(t) - \overbar{{\bf p}}\|_{\mathcal H}$ and $\|{\bf p}(t) - \overbar{{\bf p}}\|_{\ell^2}$ in semilog scale when $\mu = 3$. In both cases, the decay is exponential with respect to time.}
  \label{fig:cv_energy}
\end{figure}

\begin{remark}
As will be discussed in detail below, when $p_0(0) = 0$ (whence $p_0(t) = 0$ for all $t \geq 0$), the system \eqref{eq:law_limit_repeat}-\eqref{eq:Q_rewrite} becomes linear and $\|{\bf p}(t) - \overbar{{\bf p}}\|^2_{\mathcal H} = \sum\limits_{n=1}^\infty \frac{(p_n(t)-\overbar{p}_n)^2}{\overbar{p}_n}$. By following the argument we have presented so far (while keeping in mind that $p_0(t) = 0$ for all $t$), we can also deduce that
\begin{equation}\label{eq:main_result_p0=0}
\|{\bf p}(t) - \overbar{{\bf p}}\|^2_{\mathcal H} \leq \|{\bf p}(0) - \overbar{{\bf p}}\|^2_{\mathcal H}\,\expo^{-2\,t}.
\end{equation}
\end{remark}

\begin{remark}
Intuitively speaking, the content of Theorem \ref{thm:3} suggests that the rate of convergence of $\|{\bf p}(t) - \overbar{{\bf p}}\|^2_{\mathcal H}$ will be improved as $\mu$ increases, and such improvement is saturated when $\mu$ becomes large enough (i.e., when $\mu > 3$ according to our analytical prediction). We also remark that the larger values of $\mu$ implies faster convergence of $p_0$ to zero.
\end{remark}

To finish our analytical study of the ODE system \eqref{eq:law_limit_repeat}-\eqref{eq:Q_rewrite}, we give an alternative derivation of the bound \eqref{eq:main_result_p0=0} when the initial datum ${\bf p}(0)$ associated with \eqref{eq:law_limit_repeat}-\eqref{eq:Q_rewrite} is chosen such that $p_0(0) = 0$. As the argument to be presented below will be based on the Bakry-Émery approach \cite{bakry_diffusions_1985} and a similar procedure, although less transparent compared to the forthcoming ``brute-force'' computations, has been carried out in \cite{cao_derivation_2021} for a related model, we only provide a sketch of the proof of \eqref{eq:main_result_p0=0}.

We start with a Fokker-Planck formulation of the system \eqref{eq:law_limit_repeat}-\eqref{eq:Q_rewrite} using the notations prepared in Definition \ref{def:2}.

\begin{lemma}\label{lem:x}
Assume that ${\bf p}(t)$ is a solution of the ODE system \eqref{eq:law_limit_repeat}-\eqref{eq:Q_rewrite}, then for all $n \in \mathbb N$
\begin{equation}\label{eq:FK_nonlinear}
p'_n = Q[{\bf p}]_n =  (\mu-1)\,D^-\left(q_n\,D^+ \left(\frac{p_n}{q_n}\right)\right),
\end{equation}
where ${\bf q} \in \mathbb{P}(\mathbb N)$ is defined via $q_0 = 0$ and $q_n = \frac{(\mu-1+p_0)^{n-1}}{(n-1)!}\,\expo^{-(\mu-1+p_0)}$ for $n\geq 1$. In particular, when $p_0(0) = 0$, we have $p_0(t) = 0$ and ${\bf q}(t) = \overbar{{\bf p}}$ for all $t \geq 0$. Thus \eqref{eq:FK_nonlinear} reduces to
\begin{equation}\label{eq:FK_linear}
p'_n = Q[{\bf p}]_n =\left\{\begin{aligned}
&0, \quad n=0,\\
&(\mu-1)\,D^-\left(\overbar{p}_n\,D^+ \left(\frac{p_n}{\overbar{p}_n}\right)\right), \quad n \geq 1.
\end{aligned}
\right.
\end{equation}
\end{lemma}

\begin{remark}
We remark here that the (time-dependent) probability distribution ${\bf q}$ plays the role of a ``quasi-stationary equilibrium'' for the dynamics \eqref{eq:law_limit_repeat}-\eqref{eq:Q_rewrite}, in the sense that both ${\bf q}$ and $\overbar{{\bf p}}$ take the same form as modified Poisson distributions, but with different mean values. Meanwhile, it is clear that ${\bf q}(t) \xrightarrow{t \to \infty} \overbar{{\bf p}}$ (componentwise convergence), due to the explicit expression for $p_0(t)$.
\end{remark}

Now we employ the Bakry-Émery approach \cite{bakry_diffusions_1985} to prove the following:

\begin{theorem}\label{thm:4}
Suppose that ${\bf p}(t)$ is a solution of the system \eqref{eq:main_ODE} with ${\bf p}(0) \in \mathcal{S}_\mu$ and $\mu \in (1,\infty)$. If $p_0(0) = 0$, then
\begin{equation}\label{eq:Bakry-Emery}
\|{\bf p}(t) - \overbar{{\bf p}}\|^2_{\mathcal H} \leq \|{\bf p}(0) - \overbar{{\bf p}}\|^2_{\mathcal H}\,\expo^{-2\,t}.
\end{equation}
\end{theorem}

\begin{proof}
We present a more straightforward implementation of the Bakry-Émery method compared to the version provided in \cite{cao_derivation_2021} for the study of a related model. Let $V_n = \frac{(n-1)\,n}{2\,(\mu-1)} - n$ for $n \in \mathbb N$. We can rewrite \eqref{eq:FK_linear} as
\begin{equation*}
p'_n = (\mu-1)\,D^-\left(\frac{n}{\mu-1}\,D^+ p_n + p_n\,D^+ V_n\right), \quad n \geq 1.
\end{equation*}
A similar computation shown in the proof of \eqref{eq:IBP_linear_corollary} gives us
\begin{equation*}
\frac 12\,\frac{\dd}{\dd t} \|{\bf p} - \overbar{{\bf p}}\|^2_{\mathcal H} = \frac 12\,\frac{\dd}{\dd t} \sum\limits_{n=1}^\infty \frac{(p_n(t)-\overbar{p}_n)^2}{\overbar{p}_n} = -(\mu-1)\,\sum\limits_{n=1}^\infty \overbar{p}_n\,|D^+ r_n|^2 \leq 0,
\end{equation*}
where we defined $r_n = \frac{p_n}{\overbar{p}_n}$ for $n\geq 1$. To show exponential decay of the ``$\chi^2$ distance'' $\|{\bf p} - \overbar{{\bf p}}\|^2_{\mathcal H}$, the key idea, due to the seminar work of Bakry and Émery \cite{bakry_diffusions_1985}, is to compute the second time derivative of $\|{\bf p} - \overbar{{\bf p}}\|^2_{\mathcal H}$ and link it to its first time derivative. The following computations inherit the same spirit as the entropy methods presented in \cite{jungel_entropy_2016,toscani_entropy_1999} for other models where the state space is continuous:
\begin{equation}\label{eq:a_bit_long}
\begin{aligned}
&\frac 12\,\frac{\dd^2}{\dd t^2} \|{\bf p}-\overbar{{\bf p}}\|^2_{\mathcal H} = -2\,(\mu-1)^2\,\sum_{n\geq 1} \overbar{p}_n\,D^+ r_n\,D^+\left(\frac{D^-(\overbar{p}_n\,D^+ r_n)}{\overbar{p}_n}\right) \\
&= -2\,(\mu-1)^2\,\sum_{n\geq 1} \overbar{p}_n\,D^+ r_n\,D^+\left(\frac{\overbar{p}_{n-1}\,\Delta r_n + D^+ \overbar{p}_{n-1}\,D^+ r_n}{\overbar{p}_n}\right) \\
&= -2\,(\mu-1)^2\,\sum_{n\geq 1}\overbar{p}_n\,D^+ r_n\,\left(D^+\left(\frac{n-1}{\mu-1}\,\Delta r_n\right)-D^+\left(D^+ V_{n-1}\,D^+ r_n)\right)\right) \\
&= -2\,(\mu-1)^2\,\sum_{n\geq 1} \overbar{p}_n\,D^+ r_n\,\left(D^+\left(\frac{n-1}{\mu-1}\,\Delta r_n\right)-D^+(D^+ r_n)\,D^+ V_n - \underbrace{D^+(D^+  V_{n-1})}_{= \frac{1}{\mu-1}}\,D^+ r_n\right)\\
&= -2\,(\mu-1)^2\,\sum_{n\geq 1} \overbar{p}_n\,D^+ r_n\,\left(D^+\left(\frac{n-1}{\mu-1}\,\Delta r_n\right)-D^+(D^+ r_n)\,D^+ V_n\right) + 2\,(\mu-1)\,\sum_{n\geq 1} \overbar{p}_n\,|D^+ r_n|^2 \\
&= -2\,(\mu-1)^2\,\sum_{n\geq 1} \overbar{p}_n\,D^+ r_n\,\left(D^+\left(\frac{n-1}{\mu-1}\,\Delta r_n\right)-D^+(D^+ r_n)\,D^+ V_n\right) - \frac{\dd}{\dd t} \|{\bf p} - \overbar{{\bf p}}\|^2_{\mathcal H} \\
&\geq -\frac{\dd}{\dd t} \|{\bf p} - \overbar{{\bf p}}\|^2_{\mathcal H},
\end{aligned}
\end{equation}
where the last inequality follows from the non-negativity of
\begin{align*}
I &\coloneqq -\sum_{n\geq 1} \overbar{p}_n\,D^+ r_n\,\left(D^+\left(\frac{n-1}{\mu-1}\,\Delta r_n\right)-D^+(D^+ r_n)\,D^+ V_n\right) \\
&= -\sum_{n\geq 1} \overbar{p}_n\,D^+ r_n\,\left(\frac{1}{\mu-1}\,\Delta r_{n+1} + \frac{n-1}{\mu-1}\,D^+(\Delta r_n)- D^+(D^+ r_n)\,\underbrace{D^+ V_n}_{= \frac{n}{\mu-1}-1}\right) \\
&= -\sum_{n\geq 1} \overbar{p}_n\,D^+ r_n\,\left(D^+(D^+ r_n)-\frac{n-1}{\mu-1}\,D^+(D^+ r_{n-1})\right) \\
&= -\sum_{n\geq 1} D^+ r_n\,D^-\left(\overbar{p}_n\,D^+(D^+ r_n)\right) = \sum_{n\geq 1} \overbar{p}_n\,|D^+(D^+ r_n)|^2 \geq 0.
\end{align*}
Integrating both sides of \eqref{eq:a_bit_long} over $(t,\infty)$, and using that \[\lim_{t \to \infty} \|{\bf p}(t)-\overbar{{\bf p}}\|^2_{\mathcal H} = 0 \quad \text{and} \quad \lim_{t \to \infty} \frac{\dd}{\dd t} \|{\bf p}(t)-\overbar{{\bf p}}\|^2_{\mathcal H} = 0,\]  we obtain
\begin{equation}\label{eq:Gronwall_inequality}
\frac{\dd}{\dd t} \|{\bf p}(t)-\overbar{{\bf p}}\|^2_{\mathcal H} \leq -2\,\|{\bf p}(t)-\overbar{{\bf p}}\|^2_{\mathcal H} \quad \forall~ t\geq 0.
\end{equation}
As a result of Gronwall's inequality,
\begin{equation}\label{eq:expo_decay}
\|{\bf p}(t) - \overbar{{\bf p}}\|^2_{\mathcal H} \leq \|{\bf p}(0) - \overbar{{\bf p}}\|^2_{\mathcal H}\,\expo^{-2\,t}
\end{equation}
for all $t \geq 0$, which allows us to conclude the proof of Theorem \ref{thm:4}.
\end{proof}

\section{Conclusion and discussion}
\setcounter{equation}{0}

In this manuscript, the so-called sticky dispersion process (on a complete graph with $N$ vertices) is investigated, which can be viewed as a natural modification of the usual dispersion process introduced and studied in a number of recent works \cite{cooper_dispersion_2018,de_dispersion_2023,frieze_note_2018,shang_longest_2020}. Contrary to the purely probabilistic approach documented in the aforementioned papers, our analysis relies on a mean-field analysis where the classical kinetic theory plays a vital role \cite{villani_review_2002}. We emphasize the flexibility of a kinetic approach, which allows us to identify particles as dollars and vertices as agents (or sites), enable us to reformulate the model using languages from econophysics literature \cite{dragulescu_statistical_2000,chakraborti_econophysics_2011,kutner_econophysics_2019,pereira_econophysics_2017,savoiu_econophysics_2013}. Such reinterpretation of the sticky dispersion model under the mean-field limit as $N \to \infty$ motivates the study of the Gini index as a appropriate Lyapunov functional in section \ref{subsec:sec3.1} when $\mu \in (0,1]$.

We also provide a few directions for future work based on this manuscript. First, one can easily derive the mean-field version (as $N \to \infty$) of the usual dispersion process \cite{cooper_dispersion_2018,de_dispersion_2023}, obtaining (using the terminologies from section \ref{subsec:sec1.1}) that
\begin{equation}
  \label{eq:law_limit_dispersion}
  \frac{\dd}{\dd t} {\bf p}(t) = \mathcal{L}[{\bf p}(t)]
\end{equation}
with
\begin{equation}
  \label{eq:L}
  \mathcal{L}[{\bf p}]_n = \left\{
    \begin{array}{lll}
      -\left(\sum_{k\geq 2} k\,p_k\right)\,p_0 & \quad n=0,\\
      2\,p_2 + \left(\sum_{k\geq 2} k\,p_k\right)\,p_0 - \left(\sum_{k\geq 2} k\,p_k\right)\,p_1 & \quad n=1, \\
      (n+1)\,p_{n+1}+\left(\sum_{k\geq 2} k\,p_k\right)\,p_{n-1}- \left(n+\sum_{k\geq 2} k\,p_k\right)\,p_n, & \quad n \geq 2.
    \end{array}
  \right.
\end{equation}
The unique equilibrium solution associated with the system \eqref{eq:law_limit_dispersion}-\eqref{eq:L} can be identified easily. Indeed, the unique equilibrium solution of \eqref{eq:law_limit_dispersion}-\eqref{eq:L} in the space $\mathcal{S}_\mu$, for $\mu \in (0,1]$, is again given by the Bernoulli distribution ${\bf p}^*$ \eqref{eq:Bernoulli_repeat}. On the other hand, for $\mu > 1$, the unique equilibrium solution in the space $\mathcal{S}_\mu$, is provided by the following modified Poisson distribution $\tilde{\bf p}$ with
\begin{equation}\label{eq:modified_Poisson_repeat}
\tilde{p}_0 = 0, \quad \tilde{p}_n = \frac{\nu^{n-1}}{(n-1)!}\,\tilde{p}_1 \quad \text{for } n\geq 1,
\end{equation}
where $\tilde{p}_1 = -W_0\left(-\mu\,\expo^{-\mu}\right)$ (here $W_0(\cdot)$ denotes the principal branch of the Lambert $W$ function \cite{lambert_Observationes_1758}) and $\nu = \mu - \tilde{p}_1$. However, the analytical study of the ODE system \eqref{eq:law_limit_dispersion}-\eqref{eq:L} corresponding to the mean-field version of the classical dispersion process seems to be harder in the sense that $p_0(t)$ is no longer explicit. Moreover, we can no longer prove that the Gini index serves as a Lyapunov functional for the dynamics prescribed by \eqref{eq:law_limit_dispersion}-\eqref{eq:L} when $\mu \leq 1$. We leave the analysis of the asymptotic behavior of solutions of \eqref{eq:law_limit_dispersion}-\eqref{eq:L} to a future work.

A second direction of research may focus on the effect of ``stickiness'' in the model behavior. In our model we only allow one landlord for each house/site (recall that landlords will occupy their houses permanently and only tenants will make a move). It is also possible to allow for more than one landlord for each house. In other words, it might be interesting to assume that each house admits $m$ landlords (with $m \in \mathbb N_+$ measuring the strength of the stickiness effect) and hence the first $m$ agents visiting a given house all become ``permanent residents''. If we adopt a econophysics perspective, the economical counterpart of $m$ is the so-called wealth-flooring policy, which forbids agents whose wealth are below $m$ dollars to lose money (in the form of giving out their dollars to other agents).

\end{document}